\documentclass[12pt]{amsart}
\usepackage{amsmath,amsfonts,amssymb,graphics,mathrsfs,amsthm,eurosym,verbatim,enumerate,quotes,graphicx}
\usepackage[marginratio=1:1,tmargin=117pt, height=620pt]{geometry}
\usepackage{color}
\usepackage[all,cmtip]{xy}
\usepackage{hyperref}

\usepackage[flushmargin]{footmisc}
\usepackage[noadjust]{cite}
\newtheorem{theorem}{Theorem}[section]
\newtheorem{lemma}[theorem]{Lemma}
\newtheorem{proposition}[theorem]{Proposition}
\newtheorem{observation}[theorem]{Observation}

\newtheorem{definition}[theorem]{Definition}

\newtheorem{conjecture}{Conjecture}
\newtheorem{question}{Question}
\newtheorem{remark}[theorem]{Remark}

\newcommand{\tef}{transcendental entire function}
\newcommand{\cstartef}{transcendental self-map of $\Cstar$}
\newcommand\qfor{\quad\text{for }}
\newcommand \C{\mathbb{C}}
\newcommand \Ch{\widehat{\mathbb{C}}}
\newcommand \Cstar{\mathbb{C}^*}
\newcommand \N{\mathbb{N}}
\newcommand \R{\mathbb{R}}

\newcommand \Z{\mathbb{Z}}
\newcommand \spw{spider's web}
\newcommand \cstarspw{$\Cstar$-spider's web}
\newcommand \cstarspws{$\Cstar$-spiders' webs}

\newcommand*{\defeq}{\mathrel{\vcenter{\baselineskip0.5ex \lineskiplimit0pt
                     \hbox{\scriptsize.}\hbox{\scriptsize.}}}%
                     =}
\newcommand*{\defqe}{=\mathrel{\vcenter{\baselineskip0.5ex \lineskiplimit0pt
                     \hbox{\scriptsize.}\hbox{\scriptsize.}}}%
                     }
\def\domain{H}

\usepackage{marginnote}

\dedicatory{This paper is dedicated to our doctoral supervisors Gwyneth Stallard and Phil Rippon, \\ with thanks for all their help and support}
\begin{document}
\bibliographystyle{amsalpha}
%
%%%%%%%%%%%%%
%
% TITLE
%
%%%%%%%%%%%%%
%
\title[Spiders' webs in $\C^*$]{Spiders' webs in the punctured plane}
\author[{V. Evdoridou \and D. Mart\'i-Pete \and D. J. Sixsmith}]{Vasiliki Evdoridou \and David Mart\'i-Pete \and David J. Sixsmith}
\address{School of Mathematics and Statistics\\ The Open University\\
Milton Keynes MK7 6AA\\ UK\textsc{\newline \indent \href{https://orcid.org/0000-0002-5409-2663}{\includegraphics[width=1em,height=1em]{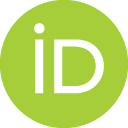} {\normalfont https://orcid.org/0000-0002-5409-2663}}}}
\email{vasiliki.evdoridou@open.ac.uk}
%\address{Department of Mathematics\\ Kyoto University\\ Kyoto 606-8502\\ Japan\textsc{\newline \indent \href{https://orcid.org/0000-0002-0541-8364}{\includegraphics[width=1em,height=1em]{Images/orcid2.png} {\normalfont https://orcid.org/0000-0002-0541-8364}}}}
\address{Institute of Mathematics of the Polish Academy of Sciences\\ ul. \'Sniadeckich 8\\
00-656 Warsaw\\ Poland\textsc{\newline \indent \href{https://orcid.org/0000-0002-0541-8364}{\includegraphics[width=1em,height=1em]{orcid2.png} {\normalfont https://orcid.org/0000-0002-0541-8364}}}}
%\email{martipete@math.kyoto-u.ac.jp}
\email{dmartipete@impan.pl}
\address{Department of Mathematical Sciences \\
	 University of Liverpool \\
   Liverpool L69 7ZL\\
   UK \newline \indent \href{https://orcid.org/0000-0002-3543-6969}{\includegraphics[width=1em,height=1em]{orcid2.png} {\normalfont https://orcid.org/0000-0002-3543-6969}}} 
\email{djs@liverpool.ac.uk}
\thanks{The first author was supported by Engineering and Physical Sciences Research Council grant EP/R010560/1. The second author was supported by the grant-in-aid 16F16807 from the Japan Society for the Promotion of Science.\vspace{3pt}\vspace{3pt}\vspace{3pt}\\ }
%
%%%%%%%%%%%%%
%
% ABSTRACT
%
%%%%%%%%%%%%%
%
\begin{abstract}
Many authors have studied sets, associated with the dynamics of a {\tef}, which have the topological property of being a spider's web. In this paper we adapt the definition of a spider's web to the punctured plane. We give several characterisations of this topological structure, and study the connection with the usual spider's web in $\C$. 

We show that there are many transcendental self-maps of $\C^*$ for which the Julia set is such a spider's web, and we construct a transcendental self-map of $\C^*$ for which the escaping set $I(f)$ has this structure and hence is connected. By way of contrast with {\tef}s, we conjecture that there is no transcendental self-map of $\C^*$ for which the \emph{fast} escaping set $A(f)$ is such a spider's web.
\end{abstract}
\maketitle
%
%%%%%
%
%%%%%
%
\section{Introduction}
Let $S$ be either the complex plane $\C$ or the punctured plane $\C^* \defeq \C\setminus\{0\}$, and suppose that $f:S\to S$ is a holomorphic function such that $\widehat{\C}\setminus S$ consists of essential singularities of $f$, where $\widehat{\C}\defeq\C\cup\{\infty\}$. We define the \textit{Fatou set} of $f$ by
$$
F(f)\defeq\{z\in S\, :\, \{f^n\}_{n\in\mathbb{N}} \text{ is a normal family in an open neighbourhood of } z\},
$$
and we let the \textit{Julia set} be its complement in $S$; that is, $J(f)\defeq S\setminus F(f)$. We use the term \emph{Fatou component} to refer to each component of $F(f)$. 

When $S=\C$, $f$ is a transcendental entire function. There is a long history of the study of the dynamics of these functions dating back to the 1920s; see \cite{bergweiler93} for a general reference. When $S=\C^*$, we say that $f$ is a \textit{transcendental self-map of~$\C^*$}. The study of the dynamics of these maps dates back to R\r{a}dstr\"om \cite{radstrom53}, and many authors have added to this subsequently; see, for example, \cite{baker87,keen,kotus,mak87,mak91}. This paper concerns the iteration of this second class of functions, but we begin by reviewing some results for the first class of functions by way of comparison. 

For a transcendental entire function $f$, the \emph{escaping set} of $f$ is defined by
$$
I(f)\defeq\{z\in\C\, :\, f^n(z)\to\infty \text{ as } n\to\infty\}.
$$
Eremenko \cite{eremenko89} showed that all components of $\overline{I(f)}$ are unbounded, and conjectured that, in fact, all components of $I(f)$ are unbounded. This conjecture, known as \emph{Eremenko's conjecture}, is still open and has motivated much research in transcendental dynamics in the past 30 years.

It is clear that Eremenko's conjecture holds in a strong way when $I(f)$ is connected. Rippon and Stallard \cite{Fast} studied several classes of {\tef}s for which this is true, and introduced the topological notion of a \emph{spider's web}. This is defined as follows.
\begin{definition}[Spider's web]\label{def:sw}
A set $E \subset \C$ is a spider's web if $E$ is connected, and there exists a sequence $(G_n)_{n\in\N}$ of bounded simply connected domains such that
\[
\partial G_n \subset E, \ G_n \subset G_{n+1}, \text{ for } n\in\N,\quad \text{ and }\quad \bigcup_{n\in\N} G_n = \mathbb{C}.
\]
\end{definition}
Roughly speaking, a set is a spider's web if it is connected, and it contains a sequence of ``loops'' surrounding each other that tend to infinity. Rippon and Stallard showed that there are many classes of {\tef}s for which $I(f)$ is a spider's web, and so Eremenko's conjecture holds for these functions. 

In fact, Rippon and Stallard studied the following two subsets of the escaping set. Suppose that $f$ is a {\tef}. First we define the \emph{maximum modulus function} $$
M(r) \defeq M(r,f) \defeq \max_{|z|=r} |f(z)|, \qfor r\geq 0.
$$
Choose $R>0$ sufficiently large that $M^n(R)\to+\infty$ as $n\to\infty$. We then define the \emph{level set}
$$
A_R(f) \defeq \{z\in\C\, :\, |f^n(z)|\geq M^n(R) \text{ for all } n\geq 0\},
$$
and the \emph{fast escaping set}
$$
A(f) \defeq \bigcup_{\ell \geq 0} f^{-\ell} (A_R(f)).
$$
It can be shown that the definition of $A(f)$ is independent of the choice of $R$. Note that $A(f)$ is completely invariant under $f$, that is, $f^{-1}(A(f))=A(f)$, but $A_R(f)$ is not.

In \cite{rippon-stallard09} it was shown that if $A_R(f)$ is a spider's web, then so is $A(f)$, and also that if $A(f)$ is a spider's web, then so is $I(f)$. Rippon and Stallard gave a number of  conditions each of which implies that $A_R(f)$ is a spider's web and so, in particular, $I(f)$ is connected; see \cite[Theorem~1.9]{Fast}, and also see \cite{DaveSW} for many other examples. They also asked the following questions.
\begin{question}
\label{q1}
Is there a {\tef} $f$ such that $A(f)$ is a spider's web but $A_R(f)$ is not?
\end{question}
\begin{question}
\label{q2}
Is there a {\tef} $f$ such that $I(f)$ is a spider's web but $A(f)$ is not?
\end{question}
Question~\ref{q1} is still open. Question~\ref{q2} was first answered in the positive by a complicated example in \cite{ripponstallard}. Later Evdoridou \cite{Vasso1} proved that for \emph{Fatou's function} $f(z) \defeq z+1+e^{-z}$ the escaping set is a spider's web but the fast escaping set is not, hence providing a simple example of a function having this property. This result was subsequently generalised in \cite{DaveVasso} to a larger class of functions.

Our goal in this paper is to transfer this study to the class of transcendental self-maps of $\Cstar$. It can be shown that such a function $f$ can be written in the form
\begin{equation}
\label{eq:cstar-form}
f(z)=z^n\exp(g(z)+h(1/z)),
\end{equation}
with $n\in\mathbb{Z}$ and $g,h$ non-constant entire functions; see \cite[p.88]{radstrom53}. %
Every holomorphic self-map $f$ of~$\C^*$ can be semiconjugated to an entire function $\tilde{f}$ by the exponential map, that is, the following diagram commutes:
\begin{equation}
\begin{split}
\label{eq:commute}
\xymatrix{
\C \ar[r]^{\tilde{f}} \ar[d]_\exp & \C\ar[d]^\exp \\
\C^* \ar[r]_f & \C^*
}
\end{split}
\end{equation}
The function $\tilde{f}$ is called a \textit{lift} of $f$.% and it is unique up to addition of integer multiples of $2\pi i$. To each holomorphic self-map of $\C^*$ we can associate an integer $\text{ind}(f)$, the \textit{index} of $f$, which is a topological invariant and satisfies that if $\tilde{f}$ is a lift of $f$, then}
%\begin{equation}
%\label{eq:lift}
%\blue{\tilde{f}(z+2\pi i)=\tilde{f}(z)+\text{ind}(f)\cdot 2\pi i,\quad \text{ for all } z\in\C.}
%\end{equation}
%\blue{Note that $\text{ind}(f)=n$ in \eqref{eq:cstar-form}.}
%Moreover, we have $n=\text{ind}(f)$ in the formulas from the three bullet points above. Although it was used before, 

We propose the following definition of a spider's web in the punctured plane.  We say that a set $X \subset \Cstar$ is \emph{bounded in $\Cstar$} if its closure in $\Ch$ does not meet $\{0, \infty\}$; otherwise we say that $X$ is \emph{unbounded in $\Cstar$}.
\begin{definition}[$\Cstar$-spider's web]
\label{def:Cstar-sw}
We say that a set $E\subset \Cstar$ is a \emph{$\C^*$-spider's web} if $E$ is connected, and there exists a sequence of domains $(G_n')_{n\in\N}$, each of which is bounded in $\Cstar$, such that
\begin{enumerate}[(a)]
\item for each $n \in \N$, the set $G_n'$ is doubly connected and separates zero from infinity;
\item $G_n' \subset G_{n+1}'$ and $\partial G_n' \subset E$, for $n \in \N$;
\item $\bigcup_{n \in \N} G_n' = \C^*$.\label{otherthing}
\end{enumerate}
\end{definition}
Our first two results indicate that this definition of a {\cstarspw} is correct. The first shows that {\cstarspws} lift in an obvious sense; see \eqref{eq:commute}.
\begin{theorem}
\label{theo:spwthesame}
If $E$ is a {\spw}, then $E' \defeq \exp(E)$ is a {\cstarspw}. Similarly, if $E'$ is a {\cstarspw}, then $E \defeq \exp^{-1}(E')$ is a {\spw}. 
\end{theorem}
The second relates to a recent result of Evdoridou and Rempe-Gillen \cite[\mbox{Theorem~1.5}]{LasseVasso}. They showed that a connected set is a spider's web if and only if it separates each point of $\C$ from $\{\infty\}$. Here if $E, X, Y \subset \Ch$, then we say that $E$ \emph{separates $Y$ from $X$} if there is an open set $U$ containing $Y$, such that the closure of $U$ in $\Ch$ does not meet $X$, and such that $\partial U \subset E$. If $Y$ is the singleton $\{z\}$, then we omit the braces and say that $E$ separates $z$ from $X$. (Note that while this is not the standard topological definition, it can be shown that this is equivalent.) Our second result shows that an equivalent result to \cite[Theorem 1.5]{LasseVasso} holds for {\cstarspws}.
\begin{theorem}
\label{theo:cstarspiderswebs}
Let $E\subset\Cstar$ be connected. Then $E$ is a {\cstarspw} if and only if it separates each point of $\Cstar$ from $\{0, \infty\}$.
\end{theorem}
The concept of a {\cstarspw} is useful if there are dynamically defined and interesting sets that have this structure. In fact, there are many holomorphic self-maps of $\Cstar$ for which is a {\cstarspw}. To see this, let $\tilde{f}$ be a lift of a holomorphic self-map $f$ of $\Cstar$. Although it had been discussed earlier, Bergweiler \cite{bergweiler95} was the first to prove rigorously that $J(f)=\exp J(\tilde{f})$. It follows from Theorem~\ref{theo:spwthesame} that $J(f)$ is a {\cstarspw} exactly when $J(\tilde{f})$ is a spider's web. To give a particular example, Osborne \cite{osborne13b} showed that $J(\tilde{f})$ is a spider's web when $\tilde{f}(z) = \sin z$. Therefore, we can deduce that $J(f)$ is a {\cstarspw} for the function
\[
f(z) \defeq \exp\left(\frac{1}{2}\left(z - \frac{1}{z}\right)\right).
\]

It follows from the above that, in a sense, there seems little further to say about transcendental self-maps of $\Cstar$ whose Julia set is a {\cstarspw}. Accordingly, we now focus on the question of whether the escaping set of a transcendental self-map of $\Cstar$ can be a {\cstarspw}. First we need to define this set. The escaping set of a transcendental self-map $f$ of $\C^*$ is given by
$$
I(f)\defeq\{z\in\C^*\, :\, \omega(z,f)\subset \{0,\infty\}\},
$$
where $\omega(z,f)\defeq\bigcap_{n\in\mathbb{N}}\overline{\{f^k(z)\, :\, k\geq n\}}$, and this closure is taken in $\Ch$. See \cite{martipete1,fagella-martipete,martipete3,newpaper} for several properties about this set. Observe that, in general, we cannot assume that $I(f) = \exp I(\tilde{f})$ whenever $\tilde{f}$ is a lift of a holomorphic self-map $f$ of $\Cstar$; see~\eqref{eq:commute}. A counter-example is given when $f(z) = z e^{z-1}$ and $\tilde{f}(z) = z + e^z - 1 + 2\pi i$. Then $\tilde{f}^n(0) = 2\pi i n$, for $n \in \N$, and so $0 \in I(\tilde{f})$. However $1 = e^0$ is a fixed point of~$f$. Forthcoming work by Mart\'i-Pete \cite{davidmisc} will consider similar examples.

We then have the following, which is the main result of this paper. This result is based on a function first considered in \cite[Example~3.3]{martipete3} which has a hyperbolic Baker domain containing a right half-plane; we leave the definition of these terms for later. See Figure~\ref{fig:hyp-baker-domain}.

\begin{theorem}
\label{thm:escaping-set-sw}
There exists $\lambda_0 > 0$ such that if $\lambda \geq \lambda_0$, then the transcendental self-map of $\Cstar$ given by
\[
f_\lambda(z) \defeq \lambda z\exp(e^{-z}/z) = \lambda z\exp\left(\frac{e^{-z}-1}{z} + \frac{1}{z}\right),
\]
has the property that $I(f_\lambda)$ is a {\cstarspw}.
\end{theorem}

\begin{figure}[ht!]
\includegraphics[width=.49\linewidth]{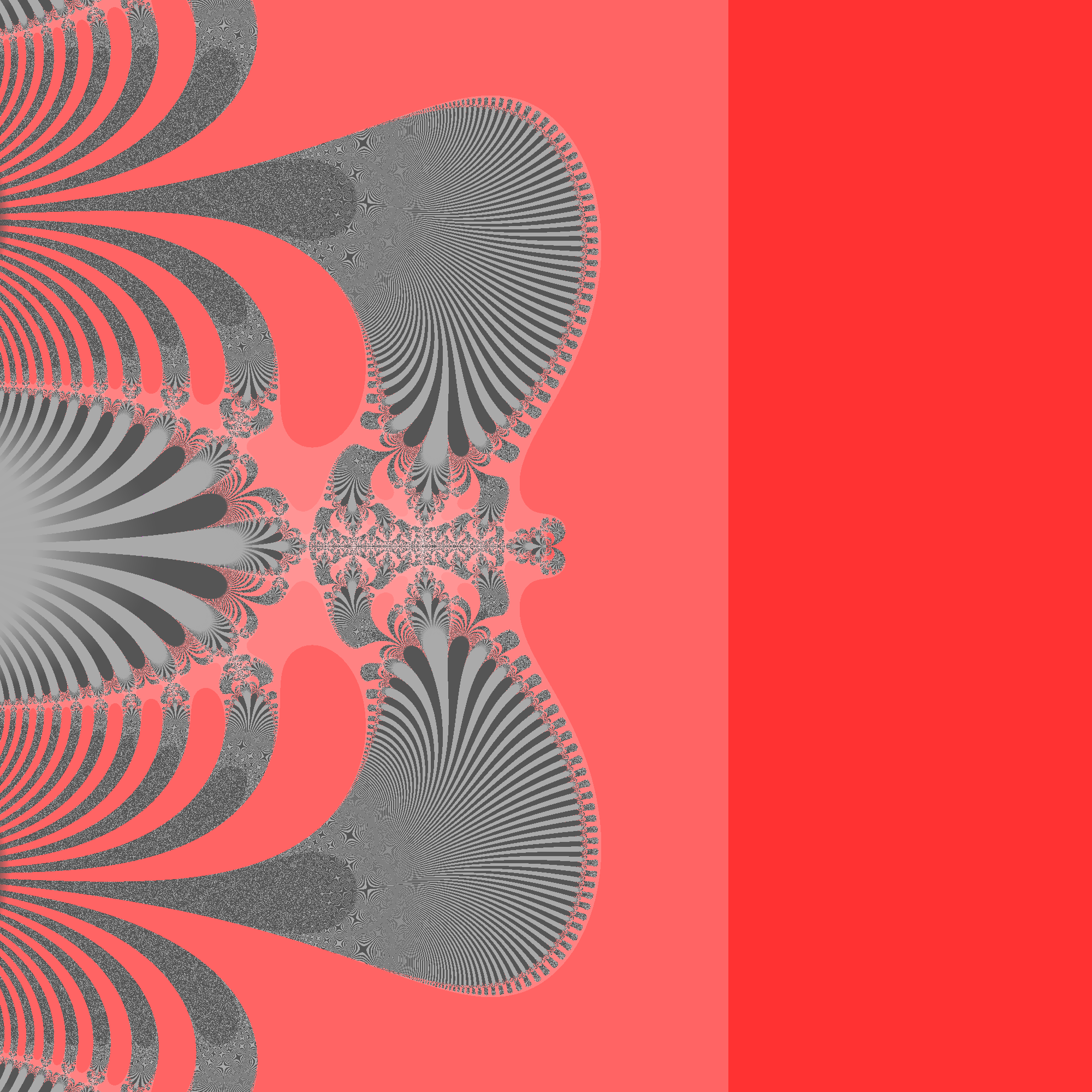}
\hspace{\fill}
\includegraphics[width=.49\linewidth]{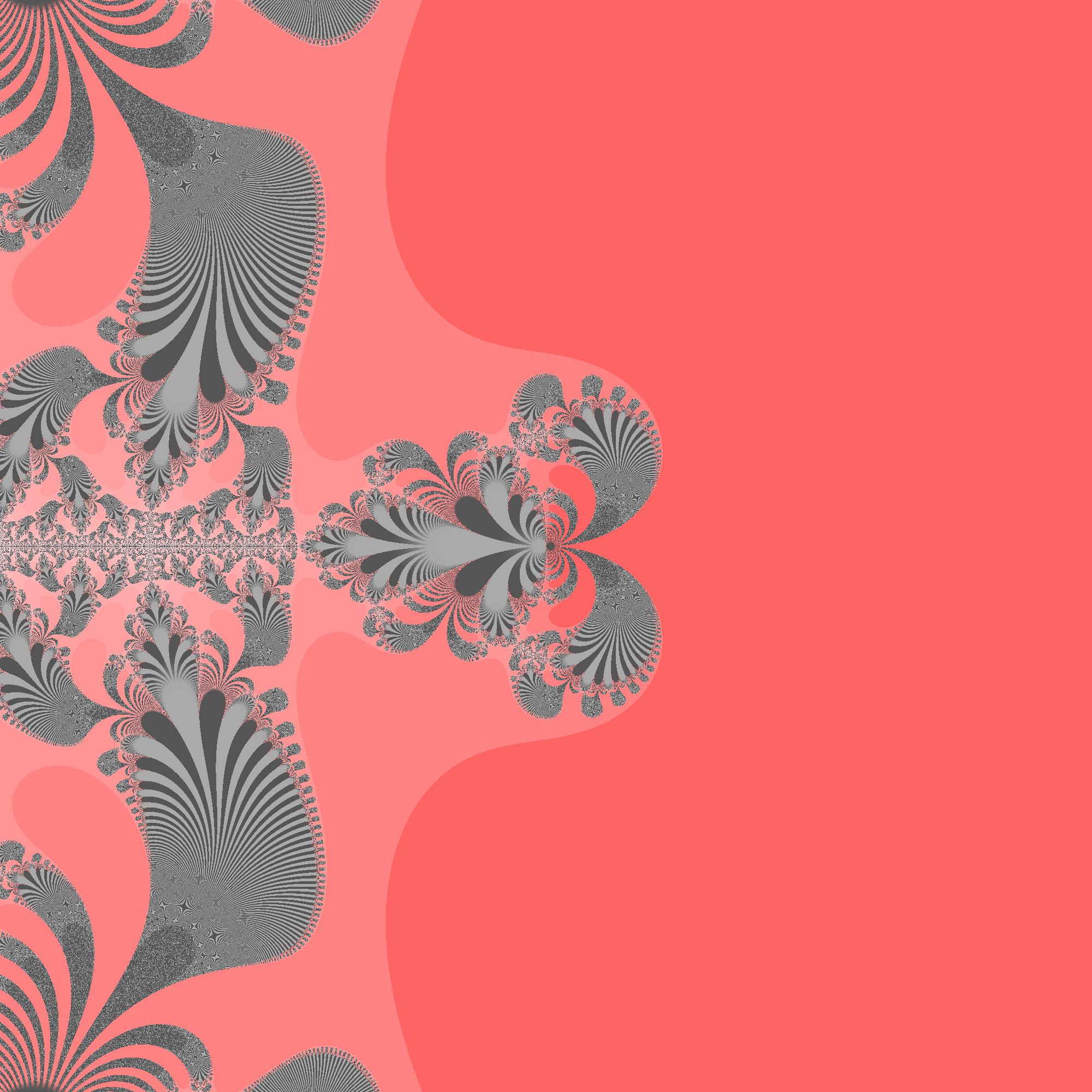}
\caption{Dynamical plane of the function $f_\lambda$ for $\lambda=32$. The Julia set is drawn in gray, and the Baker domain is coloured red. The different shades of red indicate the smallest $n$ for which $\operatorname{Re}f_\lambda^n(z) \geqslant 2$. On the left, $z\in [-6,6]+i[-6,6]$, and on the right,  $z\in [-1,1]+i[-1,1]$.}
\label{fig:hyp-baker-domain}
\end{figure}

Mart\'i-Pete also defined the fast escaping set $A(f)$ for a transcendental self-map $f$ of $\C^*$; see \cite[Definition~1.2]{martipete1}. The definition is, of necessity, quite complicated and so we omit it here. (Note that in this paper we only use the fact that the set $A(f)$ is completely invariant under $f$.) Roughly speaking, $A(f)$ contains those points which tend to $\{0, \infty\}$ eventually faster than a combination of iterates of either the maximum or the minimum modulus functions. It is then quite straightforward to see that if $f$ is the function in Theorem~\ref{thm:escaping-set-sw}, then $A(f)$ is not a {\cstarspw}; this is because there is a right half-plane in which $f$ behaves like $z \mapsto \lambda z$ and so none of these points can escape faster than the maximum modulus. Thus, this function provides a positive answer to the analogue of Question~\ref{q2} in our setting. 

It is also possible to define an analogue of the level set $A_R(f)$, though this is complicated and we refer again to \cite[Definition~1.2]{martipete1}. However, it is not difficult to show that if $f$ is a transcendental self-map of $\Cstar$, then $A_R(f)$ cannot be a {\cstarspw}. Very roughly, by way of contradiction, we let $\Gamma$ be a ``loop'' in $A_R(f)$ which surrounds and is very close to the origin. Then $f(\Gamma)$ must contain a point of very large modulus and also a point of very small modulus. Hence $f(\Gamma)$ contains a point of modulus $1$, and so the preimage of this point on $\Gamma$ cannot lie in $A_R(f)$. This gives the necessary contradiction.

Recall that all known examples of transcendental entire functions $f$ for which $A(f)$ is a spider's web have been constructed by first showing that $A_R(f)$ is a spider's web.  From the previous paragraph, it is clear that an analogous construction is not possible in the case of a transcendental self-map of $\Cstar$. In view of this fact, together with Question~\ref{q1}, we make the following conjecture.
\begin{conjecture}
If $f$ is a transcendental self-map of $\Cstar$, then $A(f)$ is not a {\cstarspw}.
\end{conjecture}

Note that our definition of a {\cstarspw} is topological. In \cite[Theorem 1.5]{Dave} it was shown that for many sets it is possible to give a dynamical definition of a spider's web; for example, if $f$ is a {\tef}, then $I(f)$ is a spider's web if and only if it separates some point of $J(f)$ from infinity. Our final result takes \cite[Theorem 1.5]{Dave} into the $\Cstar$ setting and will be the main tool to prove Theorem~\ref{thm:escaping-set-sw}.
\begin{theorem}
\label{theo:dynchar}
Suppose that $f$ is a {\cstartef}. Then $I(f)$ is a {\cstarspw} if and only if it separates some point of $J(f)$ from $\{0, \infty\}$. This statement is also true if we replace $I(f)$ with either $A(f)$ or $J(f)$.
\end{theorem}
\begin{remark}\normalfont
Note that in the case of $J(g)$, where $g$ is a {\tef}, \cite[Theorem 1.5]{Dave} requires the additional hypothesis that $g$ has no multiply connected Fatou components. Although a holomorphic self-map $f$ of $\Cstar$ can have (at most one) multiply connected Fatou component, we do not require this additional hypothesis. In fact, it follows from Propositions~\ref{prop:cstargeneral} and \ref{prop:containsspw} that if $f$ has a multiply connected Fatou component, then $J(f)$ separates no point of $J(f)$ from $\{0, \infty\}$.
\end{remark}

\noindent
\textbf{Structure.} In Section~\ref{S:cstarspiderswebs} we give the proof of Theorem~\ref{theo:cstarspiderswebs}. We then use this result in Section~\ref{S:sw-defn} to prove Theorem~\ref{theo:spwthesame}. We prove Theorem~\ref{theo:dynchar} in Section~\ref{S:dynchar}, and we use this result to prove Theorem~\ref{thm:escaping-set-sw} in Section~\ref{S:escaping-set-sw}. \\

\noindent
\textbf{Notation.} Unless otherwise stated, all topological operations such as closure and boundary are taken in $\C$. Also, if $S \subset \Cstar$, then we denote by $T(S)$ the set formed by appending to $S$ all components of $\Ch \setminus S$ that do not contain either $0$ or $\infty$. \\

\noindent
\textbf{Acknowledgments.} We are grateful to the referees for their detailed reading and much helpful feedback; in particular, for comments leading to Remark~\ref{rem:newremark}.
\section{Proof of Theorem~\ref{theo:cstarspiderswebs}}
\label{S:cstarspiderswebs}
In this section we prove Theorem~\ref{theo:cstarspiderswebs}. In fact we prove a slightly more detailed result, which is analogous to \cite[Theorem 2.10]{LasseVasso}. %, though, for reasons of simplicity, we have removed some detail that could be proved. 
Note that in this result we do not assume, initially, that $E$ is connected.
\begin{theorem}
\label{theo:cstarspiderswebsprecise}
Suppose that $E\subset\Cstar$. Then the following are equivalent.
\begin{enumerate}[(a)]
\item There is a sequence of domains $(G_n')_{n\in\N}$, as in the definition of a {\cstarspw}.\label{item:spidersweb}
\item $E$ separates every compact set in $\Cstar$ from $\{0, \infty\}$.\label{item:compactseparation}
\item $E$ separates every point $z \in \Cstar$ from $\{0, \infty\}$.  \label{item:pointseparation}
\end{enumerate}
Suppose finally that one of these equivalent conditions holds. Then $E$ is connected (that is, $E$ is a \cstarspw) if and only if $E\cup\{0, \infty\}$ is connected.
%Finally, the above also holds with ``{\icstarspw}'' replaced with ``{\ocstarspw}'' and ``$0$'' replaced with ``$\infty$''.
%Suppose now that one of these equivalent conditions holds. Then the following are equivalent:
%\begin{enumerate}[(1)]
%\item $E$ is connected (that is, $E$ is a \cstarspw);\label{item:starfullspidersweb}
%\item there is a dense collection of unbounded connected subsets of $E$\margin{Not sure about this};\label{item:starunboundedcomponents}
%\item $E\cup\{0, \infty\}$ is connected. \label{item:starconnectedwithinfinity}
%\end{enumerate}
%Finally, the above also holds with ``{\icstarspw}'' replaced with ``{\ocstarspw}'' and ``$0$'' replaced with ``$\infty$''.
\end{theorem}
\begin{proof}
It is immediate that \eqref{item:compactseparation} implies \eqref{item:pointseparation}. To see that \eqref{item:spidersweb} implies~\eqref{item:compactseparation}, let $(G_n')_{n \in \N}$ be the domains from the definition of a \cstarspw. Suppose that $X\subset \Cstar$ is compact; in particular $X$ is bounded away from $0$ and $\infty$, and so we can choose $n \in \N$, such that	$X \subset G_n'$. Then $G_n'$ is the necessary open set in the definition of a separation.
% Then $(G \setminus E)\cup X$ is a relatively open and closed subset of $(\C \setminus E)\cup X\cup\{0, \infty\}$, as required. 

Next suppose that~\eqref{item:compactseparation} holds. We claim that if $K \subset \Cstar$ is compact, and $0$ and $\infty$ lie in different components of $\Ch \setminus K$, then there is a domain $G=G(K)$ that is bounded in $\Cstar$, with $K\subset G$ and $\partial G\subset E$, and such that $\Ch \setminus G$ has exactly two components, one containing $0$ and one containing $\infty$. To see this we first, by assumption and by definition, let $U$ be an open set, bounded in $\Cstar$, such that $K \subset U$ and $\partial U \subset E$. We then set $G = T(U)$. 

%We claim  that for every nonempty, compact and connected
%     $K\subset\Cstar$, and such that $\Ch \setminus K$ has a component containing $0$ and another containing $\infty$, there is a $\Cstar$-bounded domain $G=G(K)$, with $K\subset G$ and $\partial G\subset E$ and such that $\Ch \setminus G$ has exactly two components, one containing $0$ and one containing $\infty$.
    
%To prove this, note first that as $E$ separates  $K$ and $\{0, \infty\}$, by definition there is a relatively closed and  open subset
%         \[U'\subset A \defeq (\C\setminus E)\cup  K \cup \{0, \infty\} \] 
%     such  that $K\subset U'\subset\C$.  
%      Let  $U\subset\Ch$ be open such that  $U'=U\cap  A$. Since $U'$ is relatively closed in $A$, 
%      we see that  $U$ is $\Cstar$-bounded and $\partial U\subset \Ch \setminus A \subset E$. 
%
%      Now let $V$ be the connected component
%      of $U$ containing $K$, and let $G=G(K)$ consist of $V$ 
%     together with all $\Cstar$-bounded complementary components. Clearly $\partial G\subset \partial U\subset E$. 

    So we can define a sequence $(G_n')_{n\in\N}$ of domains (from the definition of a {\cstarspw}) inductively as follows. First we 
      let $K_1$ be the circle $\{ z \in \Cstar : |z| = 1 \}$. Now suppose that $n \in \N$ and that $K_n$ has been defined. We set $G_n' = G(K_n)$, and we set $$K_{n+1}\defeq T(\{ z \in \C : 1/n \leq |z| \leq n \} \cup \overline{G_n'}).$$ The domains $G_n'$, $n\in\N$, satisfy the  requirements in  the definition of a \cstarspw, so \eqref{item:spidersweb} holds. 
    
		The remainder of the proof is very similar to that in \cite{LasseVasso}. 
    Suppose that \eqref{item:pointseparation} holds. Let $K\subset\Cstar$ be a compact set. Then for every
     $x\in K$ there is an open set $U(x)\subset\C$ that is bounded in $\Cstar$, such that $x\in U$ and $\partial U(x) \subset E$.
   
    Since $K$ is compact, there exist $k\in\N$ and points $x_1, \ldots, x_k \in K$ such that $$K\subset U\defeq \bigcup_{j=1}^k U(x_j).$$ 
       Clearly 
      \[ \partial  U \subset \bigcup_{j=1}^k \partial U(x_j)  \subset E, \]
      and $U$ is bounded in $\Cstar$. So $\partial U$  separates  $K$ from $\{0, \infty\}$. This completes the proof of the equivalence of the three conditions~\eqref{item:spidersweb} to~\eqref{item:pointseparation}.
	
For the final statement, suppose that \eqref{item:spidersweb} holds. In one direction, suppose that $E$ is connected. Since $E$ contains points of arbitrarily large and arbitrarily small modulus, we see that $\{0, \infty\}$ lies in the closure of $E$ in $\Ch$. Hence $E \cup \{0, \infty\}$ is connected. In the other direction, suppose that $E$ is not connected. Then there are disjoint open sets $U, V \subset \Cstar$ such that $E \subset U \cup V$ and $U, V$ each meet $E$. Without loss of generality, it follows that for all sufficiently large $n \in \N$, the inner boundary component of $G'_n$ lies in $V$ and the outer boundary component of $G'_n$ lies in $U$. (Here inner and outer boundary components are understood in the obvious sense, and recall that, by definition, $G'_n$ is doubly connected.) Let $V'$ denote the union of $V$ with the component of $\Ch \setminus V$ containing $0$, and let $U'$ denote the union of $U$ with the component of $\Ch \setminus U$ containing $\infty$. Then the open sets $U', V'$ disconnect the set $E \cup \{0, \infty\}$, as required.
\end{proof}
%
%%%%%
%
%%%%%
%
\section{Proof of Theorem~\ref{theo:spwthesame}} 
\label{S:sw-defn}
%\begin{proof}[]
Suppose first that $E$ is a spider's web, and that $(G_n)_{n\in\N}$ are the simply connected domains in the definition of a spider's web, which fill the plane and the boundary of each of which is in $E$. Since $E$ is connected, so is $E' = \exp(E)$. To show that $E'$ is a {\cstarspw}, it remains to construct the domains $(G'_n)_{n\in\N}$ in the definition. 

\pagebreak

For $R>0$, let $S_R$ denote the closed square whose centre is the origin
\[
S_R \defeq \{ z \in \C : \max\{|\operatorname{Re} z|, |\operatorname{Im} z|\} \leq R \}.
\]
Choose $R_1 > 2\pi$ and let $n_1 \in \N$ be sufficiently large that $S_{R_1} \subset G_{n_1}$. Then we set
\[
G'_1 \defeq T(\exp(G_{n_1})) \supset \exp(S_{R_1}) = \{ z \in \Cstar : e^{-R_1} \leq |z| \leq e^{R_1}\}.
\]
%We let $G'_{-1}$ be the component of $\Cstar \setminus \overline{\exp(G_{n_1}})$ that contains the origin in its closure, and let $G'_1$ be the union of $\exp(G_{n_1})$ with the closure of $G'_{-1}$. Since it is required by the definition, we also let $G'_0 = G'_1$.

Now, inductively suppose that $k \in \N$, and that $R_k, n_k$, and $G'_k$ have all been defined. Choose $R_{k+1} > R_k$ sufficiently large that $G'_k \subset \exp(S_{R_{k+1}})$. Let $n_{k+1} \in \N$ be sufficiently large that $S_{R_{k+1}} \subset G_{n_{k+1}}$. Then we set
\[
G'_{k+1} \defeq T(\exp(G_{n_{k+1}})) \supset \exp(S_{R_{k+1}}) = \{ z \in \Cstar : e^{-R_{k+1}} \leq |z| \leq e^{R_{k+1}}\}.
\]
It is then straightforward to check that the sequence $(G'_n)_{n\in\N}$ has the necessary properties, and this completes the proof in one direction.
%We let $G'_{-(k+1)}$ be the component of $\Cstar \setminus \overline{\exp(G_{n_{k+1}}})$ that contains the origin in its closure, and let $G'_{k+1}$ be the union of $\exp(G_{n_{k+1}})$ with the closure of $G'_{-(k+1)}$. This completes the required construction.

In the other direction, suppose that $E'$ is a \cstarspw, and that $(G'_n)_{n\in\N}$ are the domains in the definition of a \cstarspw. Let $E \defeq \exp^{-1}(E')$. We need to show first that $E$ is connected. Without loss of generality, we can assume that $G'_{n-1}$ is compactly contained in $G'_n$, for $n \geq 2$. 
%For each $n \geq 2$ let $G'_n$ be the annulus $G'_n \setminus \overline{G'_{n-1}}$. Note that $G'_n$ surrounds the origin. 
For $n \in \N$, the set $G_n \defeq \exp^{-1}(G'_n)$ is a $2\pi i$-periodic ``vertical strip'' the boundary of which lies in $E$. We claim first that $H \defeq \bigcup_{n\in\N} \partial G_n \subset E$ lies in one component of $E$. For, suppose not. It follows that there exists $n \in \N$ such that the two boundary components of $G_n$ lie in different components of $E$. In other words, there is an open set $U$, the boundary of which does not meet $E$, and that contains exactly one component of $\partial G_n$. Then $\exp(U)$ separates $E'$, which is a contradiction. 

Now suppose that $E$ is not connected. Let $T$ be a component of $E$ that does not meet $H$. Then there is an open set $U$, the boundary of which does not meet $E$, that contains $T$ and that lies in $G_n$, for some $n\in\N$. Then $\exp(U)$ lies in $G'_n$ and separates $E'$, which is again a contradiction. Thus $E$ is indeed connected.

Now suppose that $\zeta \in \C$. Since $E'$ is a {\cstarspw} it follows from Theorem~\ref{theo:cstarspiderswebsprecise} that there is an open set $U'$, bounded in $\Cstar$, that contains $\exp(\zeta)$ and the boundary of which lies in $E'$. Let $U$ be the component of $\exp^{-1}(U')$ containing $\zeta$. Then $U$ is a bounded open set, containing $\zeta$, the boundary of which lies in $E$. It follows by \cite[Theorem 2.10]{LasseVasso} that $E$ is a spider's web.
%
%%%%%
%
%%%%%
%
\section{Proof of Theorem~\ref{theo:dynchar}}
\label{S:dynchar}
In this section we require the following. This is the $\Cstar$ analogue of the blowing-up property, which is well-known for rational and  transcendental entire functions. In $\C^*$, this result is due to R\r{a}dstr\"om \cite[Theorem~4.1]{radstrom53}. Note that a transcendental self-map of $\Cstar$ has no Picard exceptional values.
\begin{lemma}
\label{lem:blow-up}
Suppose that $f$ is a transcendental self-map of~$\C^*$. Suppose also that $U\subset \C^*$ is an open set which meets $J(f)$, and that $K\subset \C^*$ is a compact set. Then there exists $n_0=n_0(K, U)\in\N$ such that $f^n(U)\supset K$ for all $n\geq n_0$.
\end{lemma}
First we prove the following, which is quite general. Here if $f : X \to X$ is a function, a set $Y \subset X$ is \emph{forward invariant} if $f(Y) \subset Y$.
\begin{proposition}
\label{prop:cstargeneral}
Suppose that $f$ is a {\cstartef}, and that the set $X \subset \Cstar$ is forward invariant. Then $X$ \emph{contains} a {\cstarspw} if and only if it separates some point of $J(f)$ from $\{0, \infty\}$.
\end{proposition}
\begin{proof}
One direction is immediate; it follows from the definitions that a set separates every point of $\Cstar$ from $\{0, \infty\}$ if it contains a {\cstarspw}.

In the other direction, suppose that $X$ separates a point of $J(f)$ from $\{0, \infty\}$. In other words, there is an open set $U$, bounded in $\Cstar$, that meets $J(f)$, and the boundary of which lies in $X$. We now inductively construct a sequence $(S_n)_{n\in\N}$ of subsets of $X$, each of which is bounded in $\Cstar$.
%, and are such that $T \defeq \bigcup_{n\in\N} S_n \subset X$ is a {\cstarspw}, which completes the proof.

For the first step in the induction, by Lemma~\ref{lem:blow-up}, there exists $p_1 \in \N$ such that \[\{ z \in \Cstar : 1/2 \leq |z| \leq 2 \} \subset f^{p_1}(U).\] We set $S_1 \defeq \partial T(f^{p_1}(U)) \subset \partial f^{p_1}(U)\subset f^{p_1}(\partial U)\subset f^{p_1}(X)\subset X$.

Now suppose that $S_n$ has been defined for some $n \in \N$. Choose $N = N(n) \in \N$ sufficiently large that
\[
S_n \subset \{ z \in \Cstar : 1/2^{N} \leq |z| \leq 2^{N} \}.
\]
Next, by Lemma~\ref{lem:blow-up} again, there exists $p_{n+1} \in \N$ such that \[\{ z \in \Cstar : 1/2^N \leq |z| \leq 2^N \} \subset f^{p_{n+1}}(U). \] We set $S_{n+1} \defeq \partial T(f^{p_{n+1}}(U))\subset X$.

Now, consider the image $f(S_n)$, for $n \in \N$. Observe that $S_n$ has two components, and if $n$ is sufficiently large one component contains only points of small modulus, and the other component contains only points of large modulus. Hence, if $n$ is sufficiently large, the image of each of these components contains a point of very small modulus and another point of very large modulus. Thus there exists $n_0 \in \N$ such that the set 
\[
S \defeq \bigcup_{n = n_0}^\infty (S_n \cup f(S_n)),
\]
is connected. Note that since $X$ is forward invariant we have that $S \subset X$.

We claim that $S$ is a {\cstarspw}. To see this, in the definition of a {\cstarspw} we let $G_n'$ be the component of $\Ch \setminus S_n$ that is bounded in $\Cstar$. It is easy to see that these sets have the necessary properties.
\end{proof} 

\begin{remark}\normalfont
\label{rem:newremark}
Note that Proposition~\ref{prop:cstargeneral} does not have an analogue for transcendental entire functions in general. To see this, suppose that $f$ is a transcendental entire function with a multiply connected Fatou component. It is well-known, see, for example, \cite{bakermconn}, that there is a Jordan curve $\gamma \subset F(f) \cap I(f)$ such that $f^n(\gamma)$ surrounds $f^{n-1}(\gamma)$, for $n \in \N$. Then $X \defeq \bigcup_{n=0}^\infty f^n(\gamma)$ is a forward invariant set that separates every point of $J(f)$ from infinity. However, $X$ does not contain a spider's web.
\end{remark}

\pagebreak

We now show that, similarly as for a {\tef}, if the escaping set contains a {\cstarspw}, then it is a {\cstarspw}.
\begin{proposition}
\label{prop:containsspw}
Suppose that $f$ is a {\cstartef}. If $I(f)$ contains a {\cstarspw}, then $I(f)$ is a {\cstarspw}. This statement is also true if we replace $I(f)$ with either $A(f)$ or $J(f)$. 
\end{proposition}
\begin{proof}
The cases of $A(f)$ and $J(f)$ are immediate, since each component of these sets is unbounded in $\Cstar$; see \cite[Theorem 1.5]{martipete1} and \cite[Theorem 2]{baker-dominguez98}, respectively.

The technique for $I(f)$ is inspired by the remark \cite[p.807]{Fast}. Suppose that $I(f)$ contains a {\cstarspw}, say $E$. We can assume that $E$ is a component of $I(f)$. Since all components of $A(f)$ are unbounded in $\Cstar$, we see that each component of $A(f)$ meets $E$. It follows that $A(f) \subset E$, since $E$ is a component of $I(f)$. We show that, in fact, $I(f) = E$, which completes the proof.

Suppose that $z \in I(f) \cap J(f)$. Since $J(f) = \partial A(f)$ \cite[Theorem 1.3]{martipete1}, it follows that $$z\in\partial A(f) \subset \overline{E}.$$ Hence $E \cup \{z\}$ is connected, and so, since $E$ is a component of $I(f)$, we can deduce that $I(f) \cap J(f) \subset E$.

It remains to show that $I(f) \cap F(f) \subset E$. First we show that if $V$ is a Fatou component that meets $I(f)$, then $\partial V$ meets $I(f)$. Clearly if $\partial V \subset I(f)$, then there is nothing to prove. Otherwise there is a point $z\in\partial V$ that is not in $I(f)$. For $n \in \N \cup \{0\}$, let $V_n$ be the Fatou component containing $f^n(V)$. Let $(G_n')_{n\in\N}$ be the sequence of domains in the definition of a {\cstarspw}. It follows that there exist $n_0, n \in \N$ such that $f^n(z) \in G_{n_0}' \cap \partial V_n$ and $f^n(V) \cap (\Cstar \setminus G_{n_0}') \ne \emptyset$. Note that every neighbourhood of either $0$ or $\infty$ contains points of $J(f)$, so $V_n$ cannot be a punctured neighbourhood of either of these points. We can deduce that $\partial V_n$ meets $\partial G_{n_0}' \subset E$. Thus $\partial V$ meets $I(f)$, as required.

Let $V$ be a Fatou component in $I(f)$, so that $\partial V$ meets $I(f)$. As $\partial V \subset J(f)$, we know that $\partial V \cap E \ne \emptyset$. The result is now immediate.
\end{proof}

Recall that $I(f)$, $A(f)$ and $J(f)$ are all completely invariant. It follows, then, that Theorem~\ref{theo:dynchar}, is an immediate consequence of Proposition~\ref{prop:cstargeneral} and Proposition~\ref{prop:containsspw}.

%
%%%%%
%
%%%%%
%
\section{Proof of Theorem~\ref{thm:escaping-set-sw}}
\label{S:escaping-set-sw}
We use the following result \cite[Lemma 4.1]{martipete1}, which is a version of \cite[Lemma 1]{slow}.
\begin{lemma}
\label{lem:slow}
Suppose that $(E_n)_{n \geq 0}$ is a sequence of compact subsets of $\Cstar$, and that $f : \Cstar \to \Cstar$ is a continuous function such that
$f(E_n) \supset E_{n+1}$, for $n \geq 0$. Then there exists $z \in E_0$ such that $f^n(z) \in E_n$, for $n \in \N$.
\end{lemma}

Recall that we need to prove that if $\lambda > 0$ is sufficiently large, then $I(f_\lambda)$ is a {\cstarspw} for the transcendental self-map of $\Cstar$ given by
\[
f_\lambda(z) \defeq \lambda z\exp(e^{-z}/z).
\]
%(In fact graph drawing suggests that $\lambda = 20$ is sufficiently large, and so we have used this for our illustrative figures.)

First we note the following. Suppose that $\lambda > 0$, and $z = x + iy$. Then
\[
\frac{e^{-z}}{z} = \frac{e^{-x}}{x^2+y^2}((x \cos y - y \sin y) - i(x \sin y + y \cos y)) \defqe \xi + \zeta i.
\] 
We obtain that
\begin{equation}
\label{eq:freal}
\operatorname{Re} f_\lambda(z) = \lambda e^\xi (x \cos \zeta - y \sin \zeta),
\end{equation}
\begin{equation}
\label{eq:fimag}
\operatorname{Im} f_\lambda(z) = \lambda e^\xi (x \sin \zeta + y \cos \zeta),
\end{equation}
and % also
%\[
%|f_\lambda(z)| = \lambda |z| \exp\left(\frac{e^{-x}}{x^2+y^2}(x \cos y - y \sin y\right) = \lambda |z| \exp\left(\frac{e^{-x} \cos(y+\arctan(y/x))}{\sqrt{x^2+y^2}}\right).
%\]
\begin{equation}
\label{fsize}
|f_\lambda(z)| = \lambda |z| \exp\left(\frac{e^{-x}}{x^2+y^2}(x \cos y - y \sin y)\right). % = \lambda |z| \exp\left(\frac{e^{-x} \cos(y+\arctan(y/x))}{\sqrt{x^2+y^2}}\right).
\end{equation}
%(Note that considering $\arg f_\lambda(z)$ does not seem to help).
We frequently use the fact, without comment, that the function $f_\lambda$ is symmetric with respect to the real line, that is, $f_\lambda(\overline{z}) = \overline{f_\lambda(z)}$ for $z\in\C^*$. %, and so $f$ is symmetric in the real line.

We say that a Fatou component $U$ is a \textit{Baker domain} if $U\subseteq I(f)$ and $U$ is periodic, that is, $f^p(U)\subseteq U$ for some $p\in\N$. In the case that $p=1$, we say that $U$ is an \textit{invariant} Baker domain. In \cite[Example 3.3]{martipete3} it was shown that for $\lambda\geq 2$, the function $f_\lambda$ has a \textit{hyperbolic} invariant Baker domain containing a right half-plane. Roughly speaking, this means that $f_\lambda$ behaves like $z\mapsto \lambda z$ on a right half-plane; we refer to the survey \cite{rippon08} for a more detailed classification of Baker domains, and see also \cite{konig,berg2001,fagella-henriksen}. For the sake of completeness, we include a proof that $f_\lambda$ has a Baker domain.

\begin{lemma}
\label{lemm:growth}
Suppose that $\lambda \geq 2$. Then $\operatorname{Re} f_\lambda(z) \geq 0.7 \lambda \operatorname{Re} z$, for $\operatorname{Re} z \geq 2$. In particular $f_\lambda$ has a Baker domain containing a right half-plane.
\end{lemma}
\begin{proof}
We have that 
\[
f_\lambda(z) = \lambda z \left(1 + \frac{1}{ze^z} + \frac{1}{2!(ze^z)^2} + \frac{1}{3!(ze^z)^3} \ldots\right).
\]
Suppose that $z = x + iy$ is such that $x \geq 2$. Then $|z e^z| \geq x \cdot e^{x} \geq 2 e^2 \geq 10$. Also $|e^z| \geq e^2 \geq 7.$ We can assume that $y \geq 0$. Hence
\begin{align*}
\operatorname{Re} f_\lambda(z) 
&= \lambda x \operatorname{Re}\left(1 + \frac{1}{ze^z} + \frac{1}{2! (z e^{z})^2} + \frac{1}{3! (ze^{z})^3} \ldots\right) - \lambda y \operatorname{Im}\left(\frac{\frac{1}{e^z} + \frac{1}{2! z e^{2z}} + \frac{1}{3! z^2 e^{3z}} \ldots}{z}\right) \\
&\geq \lambda x \left(1 - \frac{1}{10} - \frac{1}{200} - \ldots\right) -  \lambda \left(\frac{1}{7} + \frac{1}{4 \cdot 49} + \ldots \right) \\
&\geq 0.8 \lambda x - 0.2 \lambda \\
&\geq 0.7\lambda x.
\end{align*}
The claim follows.
\end{proof}

From now on, for simplicity, we write $f$ instead of $f_\lambda$, and we will assume that $\lambda \geq 2$. Next we define a closed subset $I$ of $I(f)$ and show that all its complementary components are bounded in $\Cstar$. Note that Theorem~\ref{thm:escaping-set-sw} follows from this fact. To see this, let $z \in \Cstar$ be a repelling periodic point of $f$ and let $W$ be the component of $\Cstar \setminus I$ that contains $z$. Then $W$ is a domain that contains $z$ and is bounded in $\Cstar$, with $\partial W \subset I(f)$. Thus $I(f)$ is indeed a {\cstarspw}, by Theorem~\ref{theo:dynchar}.

To define our set $I$, first we set
\[
\domain \defeq \{ z \in \Cstar : \operatorname{Re} z \geq 2 \},
\]
which is an \emph{absorbing domain}. In other words, for each point $z$ in the Baker domain, there exists $n_0\in\N$ such that $f^n(z) \in H$ for $n \geq n_0$. Then, we define
\[
I \defeq \{ z \in \Cstar : \text{for all } n\in\N, \text{either } |f^n(z)| \geq n/2 \text{ or } |f^n(z)| \leq 2/n \text{ or } f^{n+2}(z) \in \domain \}.
\]
(Note that in the definition of $I$ we allow the possibility that there are values of $n$ for which two of the conditions hold simultaneously.) Observe that $I$ is forward invariant and contains $f^{-2}(H)$; we will use these properties in the below.

It follows from Lemma~\ref{lemm:growth} that $I$ is a closed subset of $I(f)$, and so it remains to show that all components of $\Cstar \setminus I$ are bounded in $\Cstar$; see Figure~\ref{fig:spw}.

The following observations are trivial, and the proofs are omitted.
\begin{observation}
\label{obs:eventually in I}
Suppose that $z \in \Cstar \setminus I$ and $k \in \N$ are such that either $|f^n(z)|~\geq~n/2$ or $|f^n(z)| \leq 2/n$, for $1 \leq n < k$. Then $f^{k}(z) \notin f^{-2}(\domain)$.
\end{observation}
\begin{observation}
\label{obs:real line}
We have that $f(\R^+) \subset \domain$, and so $f^{-1}(\R^+) \subset f^{-2}(\domain)$.
\end{observation}

\begin{figure}[ht!]
\includegraphics[width=.49\linewidth]{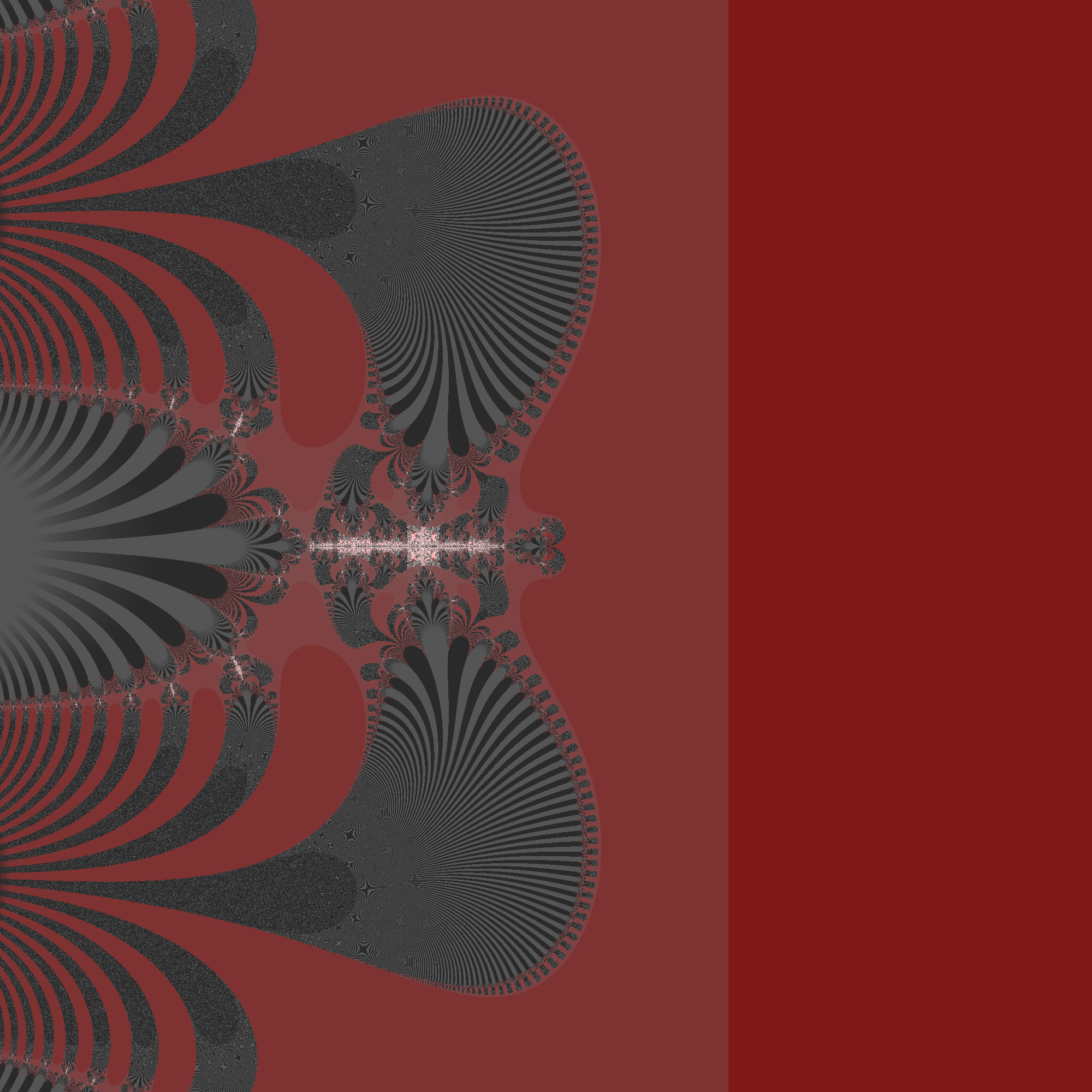}
\hspace{\fill}
\includegraphics[width=.49\linewidth]{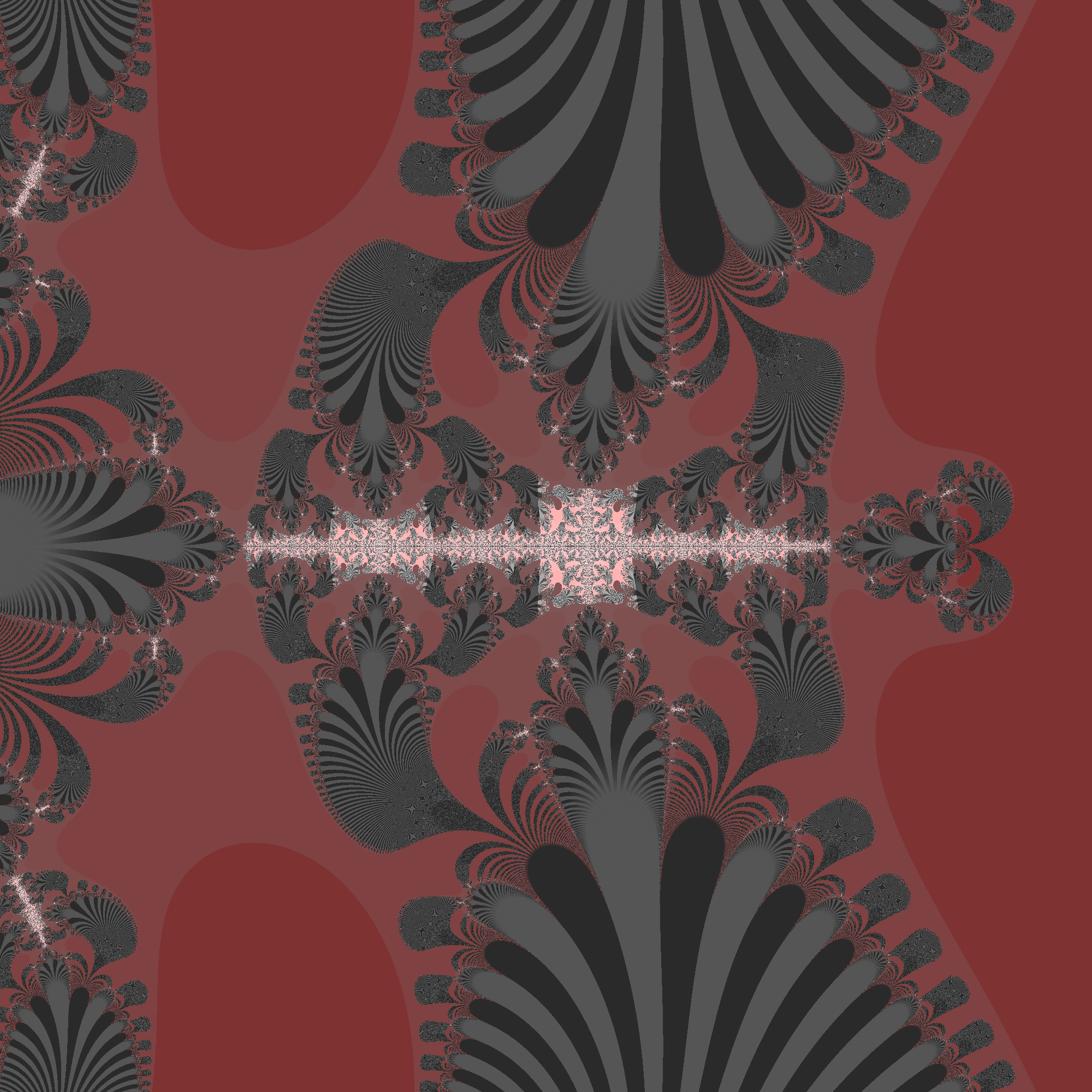}
\caption{The complement of the closed set $I\subset I(f_\lambda)$ for $\lambda=32$ is depicted in light colours. On the left, $z\in [-6,6]+i[-6,6]$. On the right, a zoom of the largest component containing the two repelling fixed points on the negative real axis; here $z\in [-3.5,0.5]+i[-2,2]$. Note that $I$ is bounded away from zero, though this is hard to show in a figure.}
\label{fig:spw}
\end{figure}

The proof that all components of $\Cstar \setminus I$ are bounded in~$\Cstar$ is long and complicated, so we begin with a very rough sketch of the strategy. We begin by supposing, by way of contradiction, that there is a component, $X$, of~$\Cstar \setminus I$ that is unbounded in~$\Cstar$. Since $X$ is open, this implies that there is a long curve, $\Gamma_1$ say, contained in $X$. Note that here, by \textit{long} we mean that $\Gamma_1$ contains points the ratio of whose moduli is large; see Lemma~\ref{lemm:longcurvesnearzero} for a more precise statement.
%is near infinity and is long in the usual sense, or $\Gamma_1$ is near the origin and its image under the map $z \mapsto 1/z$ is long in the usual sense.

By studying the set $f^{-2}(\domain)$, and by Observation~\ref{obs:eventually in I}, we show that $\Gamma_1$ must contain a subcurve, $\Gamma_1'$, that lies in one of a collection of channels in which either $|f|$ is very large or $|f|$ is very small; these sets are illustrated in Figure~\ref{fig:channels}. We are then able to deduce that $\Gamma_2 = f(\Gamma_1')$ is a long curve, in the sense mentioned above, which also does not meet $f^{-2}(\domain)$. Hence we can apply the process above to $\Gamma_2$ to obtain another curve $\Gamma_3$. We then iterate this process, and hence, using Lemma~\ref{lem:slow}, prove the existence of a point $z \in \Gamma_1 \cap I$, which is a contradiction. This completes the sketch of our proof.

It is clear from the definition of $I$, together with Observation~\ref{obs:eventually in I}, that the preimage $f^{-2}(\domain)$ plays an important role. Hence, in view of Observation~\ref{obs:real line}, we begin by considering the preimage of the positive real axis (see Figure~\ref{fig:curves-origin}).% These are the red lines in Figure~\ref{f4}.

%\begin{figure}
%\includegraphics[width=14cm,height=8cm]{Graphs/redcurves.png}
%\caption{Red lines are the preimages of the positive real axis $f^{-1}(\R^+)$. Note the curve $\Gamma$ from Lemma~\ref{lem:preimages of the real line} which crosses the real line. The unit disc is also shown, in black.}\label{f4}
%\end{figure}

\begin{lemma}
\label{lem:preimages of the real line}
Set $V = f^{-1}(\R^+)$. Then $V$ consists of the following curves, which we call the \emph{curves in $V$}:
\begin{itemize}
\item The positive real axis $\R^+$.
%\item An analytic curve $\Gamma$ from $0$ to $0$ that lies inside the unit disc and is symmetric in the real axis, which it crosses once near the point $0.8$.
\item A collection of curves from $0$ to $0$ that lie in the unit disc. %, each of which lies either in the upper half-plane or in the lower half-plane.
\item A collection of curves from $\infty$ to $\infty$ that lie outside the unit disc. Moreover, $\textup{Re}\,z\to -\infty$ and $|\textup{Im}\,z|$ is bounded as $z\to\infty$ on these curves.
\end{itemize}
Exactly one of the curves in $V$, apart from $\R^+$ itself, meets the positive real axis and none meet the negative real axis. Finally, there is a number $t >0$ such that if $z_1$ and $z_2$ are two points of the same curve in $V$, then $|\operatorname{Im} z_1 - \operatorname{Im} z_2| < t.$
\end{lemma}
\begin{remark}\normalfont
\label{rem:r1}
By considering the critical points of $f$, we can in fact show something slightly stronger. Within the unit disc, $V$ contains a curve $\Gamma$ from $0$ to $0$ that is symmetric in the real axis, which it crosses once near the point $0.8$; see Figure~\ref{fig:curves-origin}. All the other curves in $V$ are disjoint, and lie in a left half-plane. However, the proof of these facts is complicated and is omitted, as this extra detail is not required for the proof of our result.
\end{remark}

%\begin{figure}[ht!]
%\includegraphics[width=.65\linewidth, height=.6\linewidth]{Graphs/unit-circle.png}
%\caption{The curves in $V$ that lie inside the unit disc.}
%\label{fig:curves-origin}
%\end{figure}

\begin{figure}[ht!]
\includegraphics[width=.49\linewidth, height=.45\linewidth]{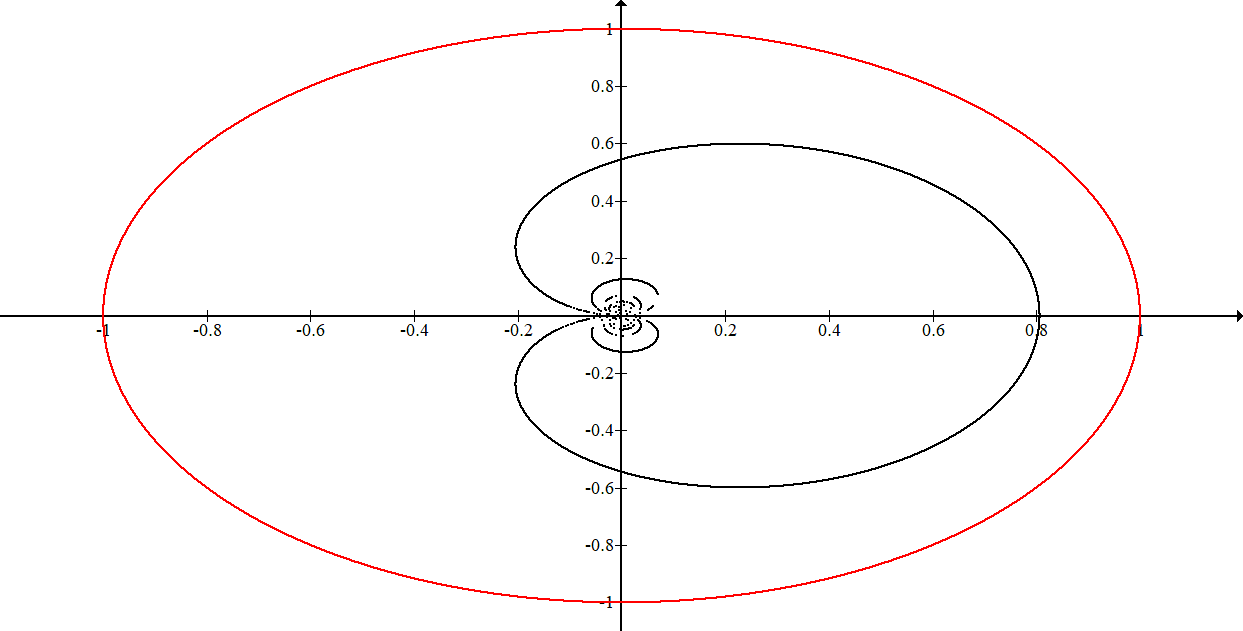}
\hspace{\fill}
\includegraphics[width=.45\linewidth, height=.45\linewidth]{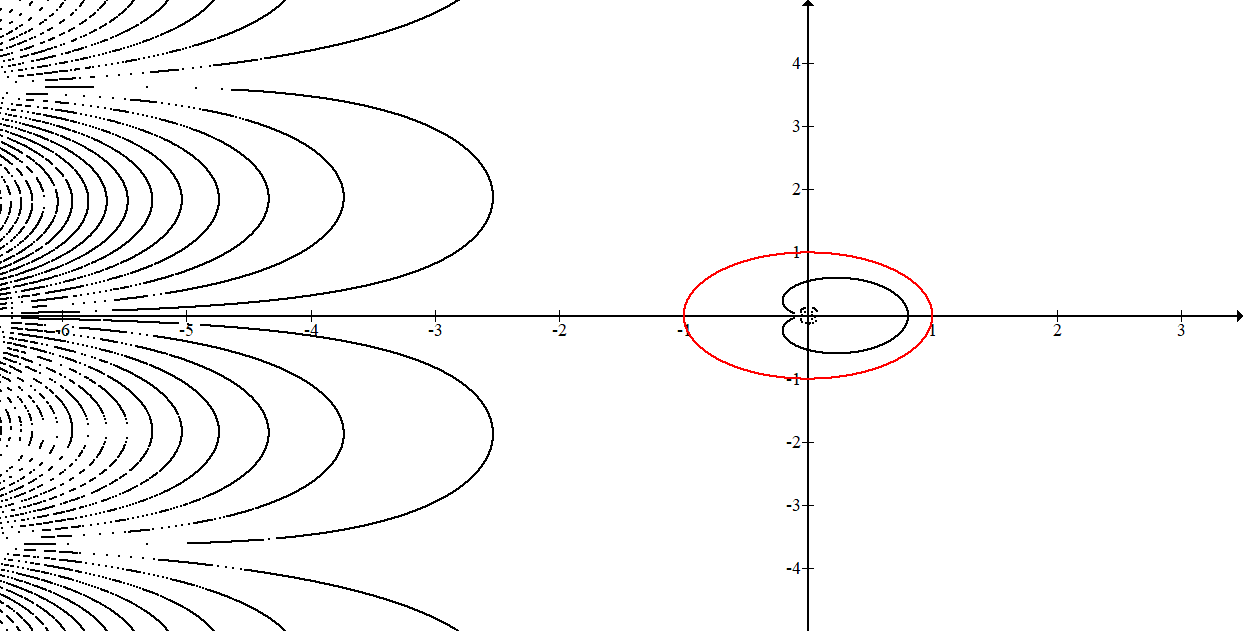}
\caption{Two images of the curves in $V$, showing the unit circle in red.}
\label{fig:curves-origin}
\end{figure}

\begin{proof}[Proof of Lemma~\ref{lem:preimages of the real line}]
Suppose that $z = x + iy \in V$, and we can assume that $y \geq 0$. It follows from \eqref{eq:fimag} that
\begin{equation}
\label{eq:inV}
y \cos \left(\frac{e^{-x}}{x^2+y^2}\left(x \sin y + y \cos y\right)\right) = x \sin \left(\frac{e^{-x}}{x^2+y^2}\left(x \sin y + y \cos y\right)\right).
\end{equation}
(Note that the solutions to this equation give preimages of both the positive real line and the negative real line.)

Observe that $V$ is necessarily a union of analytic curves which run from $\{0, \infty\}$ to $\{0, \infty\}$. One of these curves is the positive real axis. We claim that this is the only curve in $V$ that meets the unit circle $$\{ z \in \Cstar\, :\, |z| = 1 \}.$$ To prove this, suppose that $x = \cos \theta$ and $y = \sin \theta$ are such that $x + iy$ is in $V$. From \eqref{eq:inV} we obtain that
\[
\tan \theta = \tan \left(e^{-\cos \theta}\left(\cos \theta \sin(\sin \theta) + \sin \theta \cos(\sin \theta)\right)\right),
\]
and hence, by the compound angle formula, that
\[
\theta + n \pi = e^{-\cos \theta}\sin (\sin \theta + \theta), \qfor n \in \Z.
\]
It can then be checked (numerically or by graph drawing) that the only solutions to this equation are when $\theta$ is a multiple of $\pi$, as required. (Note that the point corresponding to $\theta = \pi$ lies on the negative real axis, which is a preimage of the negative real axis.) 

It follows that the only curve in $V$ which crosses the unit circle is the positive real axis. Hence all other curve in $V$ run either from $0$ to $0$, or from $\infty$ to $\infty$. 

To prove the penultimate claim of the lemma, note first that the negative real axis maps under $f$ to the negative real axis. Next we note that the critical points of $f$ occur when
\[
e^z = 1 + \frac{1}{z}.
\]
It can then be checked that there is only one critical point on the positive real axis, and this is a simple critical point; this is the point close to $0.8$ mentioned in Remark~\ref{rem:r1} and visible in Figure~\ref{fig:curves-origin}. The penultimate claim of the lemma follows.

It remains to consider the behaviour of $V$ near infinity. Observe from \eqref{eq:inV} that, apart from the positive real line, no curve in $V$ can meet the set of points $z = x + iy$ where
\begin{equation}
\label{eq:cantcross}
x \sin y + y \cos y = 0.
\end{equation}
%which is equivalent to
%\[
%y/x = -\tan y.
%\]

\pagebreak

In polar coordinates this gives
\[
\theta + n \pi = -r \sin \theta, \qfor n \in \Z.
\]
For large values of $n$, this is approximately the horizontal line
\[
y = r \sin \theta = n \pi, \qfor n \in \Z,
\]
and so no curve in $V$ can tend to infinity by increasing imaginary part. In particular, this gives the final claim of the lemma. It also implies that we can consider separately the solutions to \eqref{eq:inV} in the two cases that $x$ is large and positive, and $x$ is large and negative.

First, suppose that $x$ is large and positive. Then the terms $$\frac{e^{-x} x \sin y}{x^2+y^2} \quad\text{and}\quad \frac{e^{-x} y \cos y}{x^2+y^2}$$ are both small (independently of the size of $y$). Thus if \eqref{eq:inV} holds, then $y$ is small compared to $x$. It follows that, for $x$ large and positive, the solutions to \eqref{eq:inV} are approximately when
\[
y = e^{-x}\sin y,
\]
from which we can deduce that $y=0$ is the only solution to \eqref{eq:inV} when $x$ is large and positive. It follows that the curves from $\infty$ to $\infty$ in $V$ must tend to infinity via large, negative real parts, as claimed. This completes the proof.
\end{proof}

We will use our next result to prove that certain curves that do not meet $f^{-2}(\domain)$ must contain points of large (negative) real part.
\begin{lemma}
\label{lem:imag parts}
Suppose that $\lambda \geq 32$. Then there exists $t' > 0$ with the following property. Suppose that $\Gamma$ is a connected set that does not meet $f^{-2}(\domain)$, and that $z_1, z_2 \in \Gamma$. Then $|\operatorname{Im} z_1 - \operatorname{Im} z_2| < t'$.
\end{lemma}
%
%Now we describe the preimage $f^{-2}(\domain)$ in a little more detail.
%\begin{lemma}
%\label{lem:subsets of I}
\begin{proof}
To prove this result we, in fact, prove the following. We claim that if $\lambda \geq 32$, then for all $n \in \Z$ with $|n|$ sufficiently large, there exists $x_n < 0$ such that the horizontal half-line $$\{z= x \pm 2n\pi i : x  \geq x_n \}$$ lies in $f^{-2}(\domain)$ and meets a component of $f^{-1}(\R^+)$; see Figure~\ref{f1}, which illustrates the case $n=1$. 
\begin{figure}
\includegraphics[width=\linewidth]{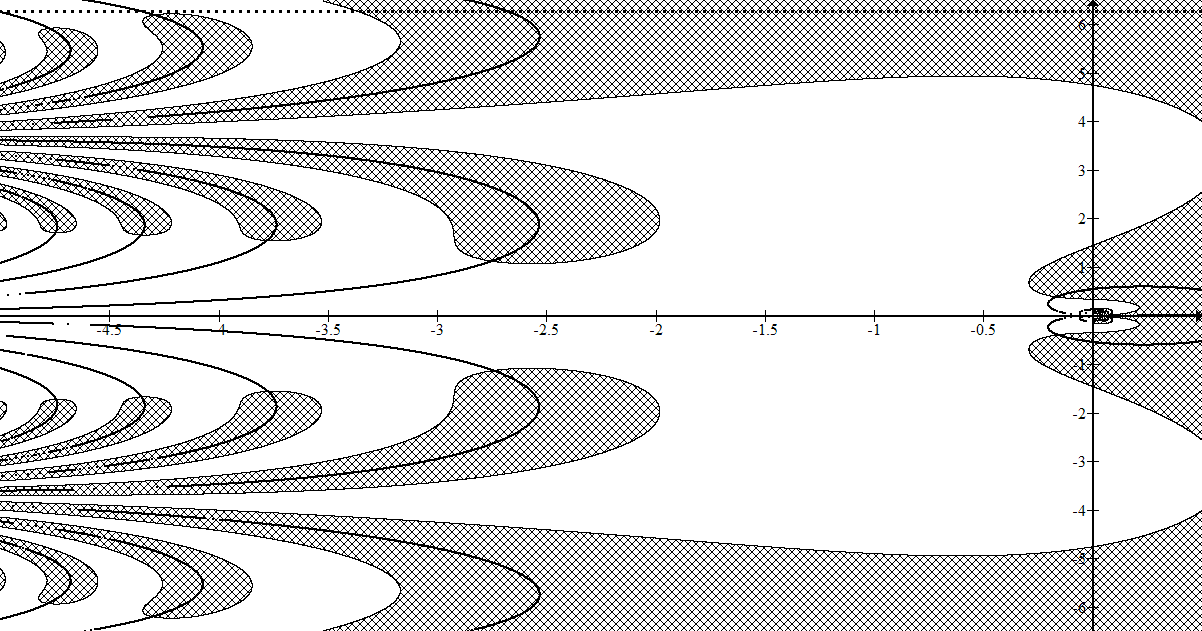}
\caption{Once again, the hatched area is $f_{32}^{-1}(\domain)$ and the black lines are the preimages of the positive real axis. Note, on the left, the curves that are approximated by the sets $A_n'$ from Lemma~\ref{lemm:longcurvesnearinfinity}. Note also that the line $y = 2\pi$, shown dotted, meets a preimage of the positive real axis inside $f_{32}^{-1}(\domain)$; see Lemma~\ref{lem:imag parts}.}\label{f1}
\end{figure}
Note that the lemma follows from this claim, by Observation~\ref{obs:real line} and the final part of Lemma~\ref{lem:preimages of the real line}.
%\end{lemma}
%\begin{remark}\normalfont
%Note that, taken in conjunction with the final part of Lemma~\ref{lem:preimages of the real line}, this result implies that there is a uniform upper bound on the difference between the imaginary parts of two points which lie in a connected set which does not meet $f^{-2}(\domain)$. This is, in fact, our goal with this result.
%\end{remark}
%(Note that graphs suggest that with $\lambda = 20$ this lemma holds for all $n \in \Z$.)
%\begin{proof}[Proof of Lemma~\ref{lem:subsets of I}]

To prove this, suppose that $y = 2n\pi$, where $n \in \N$ is large, and $z = x+ iy$. Then, by \eqref{eq:freal},
\[
\operatorname{Re} f(z) = \lambda \exp\left(\frac{e^{-x}x}{x^2+y^2}\right) \left(x \cos\left(\frac{e^{-x}y}{x^2+y^2}\right) + y \sin\left(\frac{e^{-x}y}{x^2+y^2}\right)\right).
\]

\pagebreak

We now consider three ranges of values of $x$. Recall that $\lambda \geq 32$. Suppose first that $x \geq 1$. If $n > 0$ is sufficiently large, then

\[
\operatorname{Re} f(z) \geq \lambda \left(x \cos\left(\frac{e^{-x}y}{x^2+y^2}\right) + y \sin\left(\frac{e^{-x}y}{x^2+y^2}\right)\right) \geq \frac{\lambda x}{2} \geq 2.
\]

Suppose next that $0 \leq x < 1$. If $n > 0$ is sufficiently large, then
\[
\operatorname{Re} f(z) \geq \lambda \left(x \cos\left(\frac{e^{-x}y}{x^2+y^2}\right) + y \sin\left(\frac{e^{-x}y}{x^2+y^2}\right)\right) \geq \frac{\lambda y}{2} \cdot \frac{e^{-x}y}{x^2+y^2} \geq \frac{\lambda}{4e} \geq 2.
\]

Suppose finally that $-(\log y + \log 2) \leq x < 0$. It follows that if $n > 0$ is sufficiently large, then, since $|x/y|$ can be assumed to be close to zero, we have 
\[
0 < \frac{e^{-x}y}{x^2+y^2} + \arctan(x/y) < 2,
\]
from which we can deduce that
\[
\sin\left(\frac{e^{-x}y}{x^2+y^2} + \arctan(x/y) \right) > \frac{e^{-x}y}{4(x^2+y^2)}.
\]

It follows that
\begin{align*}
\operatorname{Re} f(z) &= \lambda \exp\left(\frac{e^{-x}x}{x^2+y^2}\right) \left(x \cos\left(\frac{e^{-x}y}{x^2+y^2}\right) + y \sin\left(\frac{e^{-x}y}{x^2+y^2}\right)\right) \\
                       &= \lambda \exp\left(\frac{e^{-x}x}{x^2+y^2}\right) \cdot \sqrt{x^2 + y^2} \cdot \sin\left(\frac{e^{-x}y}{x^2+y^2} + \arctan(x/y) \right) \\
											 &\geq\lambda \exp\left(\frac{e^{-x}x}{x^2+y^2}\right) \cdot \frac{e^{-x}y}{4\sqrt{x^2+y^2}} \\
											 &\geq\lambda \exp\left(\frac{2xy}{x^2+y^2}\right) \cdot \frac{y}{4\sqrt{x^2+y^2}} \\
											 &\geq \frac{\lambda}{16} \\
											 &\geq 2.
\end{align*}
%where we have used the fact that the term in the sine function in the second line can be assumed to be between $\frac{e^{-x}y}{2(x^2+y^2)}$ and $2$.

In other words, if $n$ is sufficiently large, then the whole line segment $$\{ z = x + 2n\pi i : -(\log 2n\pi + \log 2) \leq x \}$$ maps into $\domain$. The proof is complete if we can show that there is a point of $f^{-1}(\R^+)$ somewhere on this line. When $y = 2n\pi$, equation \eqref{eq:inV} gives
\[
y \cos \left(\frac{e^{-x}y}{x^2+y^2}\right) = x \sin \left(\frac{e^{-x}y}{x^2+y^2}\right).
\]
It follows by a calculation that, when $|n|$ is large, this has a solution close to $$x = -(\log y + \log(\pi/2)).$$ This completes the proof.
\end{proof}
In the remainder of this paper we shall assume that $\lambda \geq 32$, so that the conclusions of Lemma~\ref{lem:imag parts} hold. \\

We now define some sets, which we call \emph{channels}, where the modulus of $f$ is either very large or very small. Suppose that $R \geq 1$ and, for definiteness, fix $\epsilon_0 = 1/4$. We define the following disjoint domains (which depend on $R$ only);
\[
C^+(R) \defeq \{ z = x + iy \in \Cstar : x > 0, \ |y| < \epsilon_0 |x|, \ |z| < 1 / R \};
\] 
\[
C^-(R) \defeq \{ z = x + iy \in \Cstar : x < 0, \ |y| < \epsilon_0 |x|, \ |z| < 1 / R \};
\] 
and for each $n \in \Z$,
\[
C_n(R) \defeq \{ z = x + iy \in \Cstar : x < -R (|n|+1), \ |y - n\pi| < \epsilon_0 \}.
\] 
Finally we set
\[
C(R) \defeq C^+(R) \cup C^-(R) \cup \bigcup_{n \in \Z} C_{n}(R).
\]
See Figure~\ref{f3} for a rough schematic of these sets. Note that if $1 \leq R < R'$, then $C(R') \subset C(R)$.

\begin{figure}
\includegraphics[width=\linewidth]{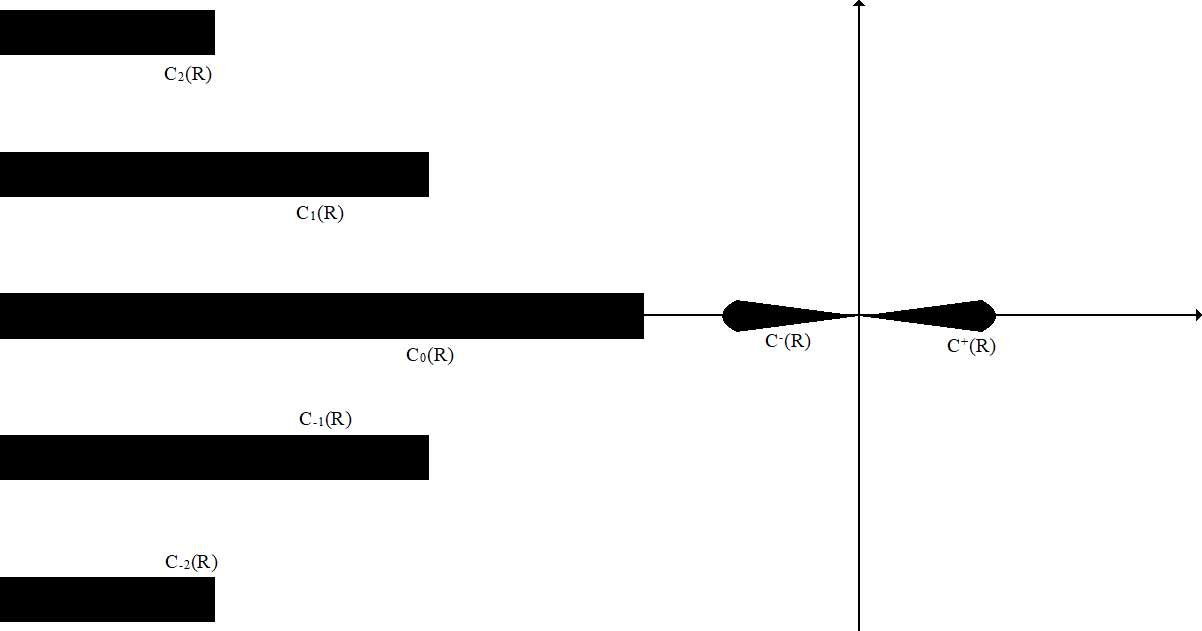}
\caption{An approximate schematic of the channels that make up $C(R)$.}\label{f3}
\label{fig:channels}
\end{figure}

We then have the following which shows that, in a suitable sense, the function $f$ ``blows up'' in the channels. Note that in this result a much faster rate of growth is possible, but this is not needed.
\begin{lemma}
\label{lemm:bigplaces}
There exists $K>1$ such that for all $L > 1$, there exist $R_0 > 1$ with the following property. Suppose that $z, z_0, z_1$ all lie in the same component, $W$ say, of $C(R_0)$. Suppose also that $|z_1/z_0| \geq K$. Then:
\begin{enumerate}
\item If $W = C^+(R_0)$, then \label{cplus}  $$|f(z)| \geq \frac{L}{|z|} \quad\text{and}\quad \left|\frac{f(z_0)}{f(z_1)}\right| \geq LK.$$
\item If $W = C^-(R_0)$, then \label{cminus} $$|f(z)| \leq \frac{|z|}{L} \quad\text{and}\quad \left|\frac{f(z_1)}{f(z_0)}\right| \geq LK.$$
\item If $W = C_{2n}(R_0)$, where $n \in \Z$, then \label{ceven}  $$|f(z)| \geq L|z| \quad\text{and}\quad \left|\frac{f(z_1)}{f(z_0)}\right| \geq LK.$$
\item If $W = C_{2n+1}(R_0)$, where $n \in \Z$, then \label{codd} $$|f(z)| \leq \frac{1}{L|z|} \quad\text{and}\quad \left|\frac{f(z_0)}{f(z_1)}\right| \geq LK.$$
\end{enumerate}
\end{lemma}
\begin{proof}
We begin by fixing values of $t>0$ and $K>1$ sufficiently large that if $z_0 = x_0 + i y_0$ and $z_1 = x_1 + iy_1$ both lie in the same component of  $C(t)$ and $|z_1/z_0| \geq K$, then $|x_1| \geq 8|x_0|$. 

Now suppose that $L>1$. Consider first case \eqref{cplus}. Suppose that $z = x + iy$, $z_0 = x_0 + i y_0$ and $z_1 = x_1 + iy_1$ are all points of $C^+(R_0)$ with $|z_1/z_0| \geq K$. 

Then, provided that $R_0>t$ is sufficiently large,
\[
\frac{2}{x} \geq \frac{e^{-x}}{x^2+y^2}(x \cos y - y \sin y) \geq \frac{1}{2x}. 
\]
Thus, by \eqref{fsize}, if $R_0>t$ is sufficiently large, then
\[
|f(z)| \geq \lambda x \exp\left(\frac{1}{2x}\right) \geq \frac{L}{x} \geq \frac{L}{|z|}.
\]
Similarly, if $R_0>t$ is sufficiently large, then
\[
|f(z_0)| \geq \lambda x_0 \exp\left(\frac{1}{2x_0}\right) \geq 2LK \lambda x_1 \exp\left(\frac{2}{x_1}\right) \geq LK |f(z_1)|.
\]

Case \eqref{cminus} is very similar, and is omitted. For case \eqref{ceven}, suppose that $z = x + iy$, $z_0 = x_0 + i y_0$ and $z_1 = x_1 + iy_1$ are all points of $C_n(R_0)$, for some even $n \in \N$, with $|z_1/z_0| \geq K$. Recall that $|y-n\pi| < 1/4$. Then, provided that $R_0>t$ is sufficiently large we can assume that $|x|$ is large compared to $y$, and so 
\[
\frac{2e^{-x}}{|x|} \geq \frac{e^{-x} \cos(y+\arctan(y/x))}{\sqrt{x^2+y^2}} \geq \frac{e^{-x}}{2|x|}. 
\]
The remainder of the proof of this case is very similar to the remainder of the proof of case \eqref{cplus}. Finally, case \eqref{codd} is very similar to the proof of the case \eqref{ceven} and is omitted.
\end{proof}
Our next two results are somewhat similar. They both say, roughly, that a suitable curve that does not meet $f^{-2}(\domain)$ must contain a subcurve which lies in the channels. Each result requires us to locate the preimages of $\R^+$, and so find the solutions to \eqref{eq:inV}. We do this by estimating these solutions; first near the origin, and then for points of large, negative imaginary part.
\begin{lemma}
\label{lemm:longcurvesnearzero}
There exists $r_0 > 0$ and $L_0 > 1$ with the following property. Suppose that $R \geq r_0$, that $L \geq L_0$, that $K>1$, and that $z_0, z_1 \in \Cstar \setminus f^{-2}(\domain)$ are two points of modulus less than $1/R$, such that $|z_1/z_0| \geq LK$. Suppose also that $\gamma \subset \Cstar \setminus f^{-2}(\domain)$ is a curve from $z_1$ to $z_0$. Then there is a curve $\gamma' \subset \gamma \cap C(R)$ joining endpoints $z_0'$ and $z_1'$, with $|z_1'/z_0'| \geq K$. More precisely, $\gamma'$ lies in either $C^+(R)$ or $C^-(R)$.
\end{lemma}
\begin{proof}
Recall, from Observation~\ref{obs:real line}, that $f^{-1}(\R^+) \subset f^{-2}(\domain)$. Recall also that if the point $x + iy \in f^{-1}(\R^+)$, then the pair $x, y$ is a solution to \eqref{eq:inV}. We begin, therefore, by estimating the solutions to \eqref{eq:inV} near the origin; see Figure~\ref{f2}. 

\begin{figure}
\includegraphics[width=\linewidth]{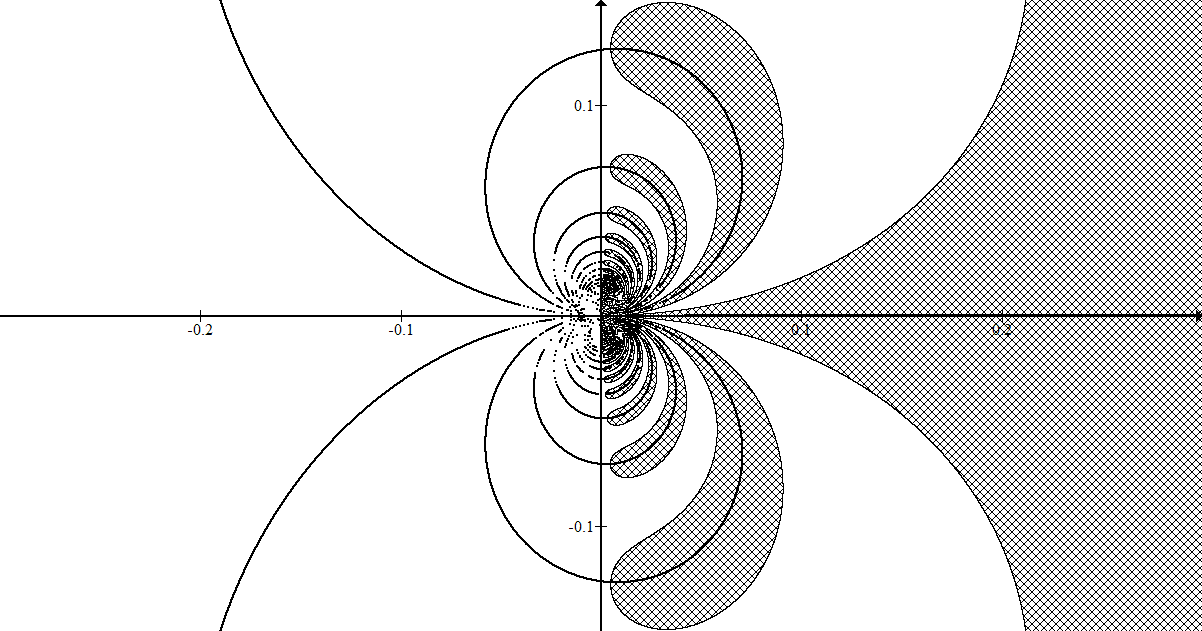}
\caption{In this figure the hatched area is $f_{32}^{-1}(\domain)$ and the black lines are the preimages of the positive real axis. Notice that, as calculated in Lemma~\ref{lemm:longcurvesnearzero}, the preimages of the positive real axis are approximately circles. Note that there is some distortion near the origin.}\label{f2}
\end{figure}

By the penultimate part of Lemma~\ref{lem:preimages of the real line},  the solutions to \eqref{eq:inV} near the origin do not cross the real axis. It follows that, near the origin, solutions to \eqref{eq:inV} can be written in the form
\begin{equation}
\label{eq:est}
\frac{e^{-x}}{x^2 + y^2}(x \sin y  + y \cos y) = n \pi + O(1), \qfor n \in \Z \text{ as } |n| \rightarrow \infty,
\end{equation}
where the $O(1)$ term is equal to $\arctan(y/x)$. By estimating the functions in \eqref{eq:est} near the origin, we obtain that in a small neighbourhood of the origin
\begin{equation}
\label{eq:est2}
x^2 + \left(y - \frac{1}{2n\pi}\right)^2 = \left(\frac{1}{2n\pi}\right)^2 + (x^2+y^2) \cdot O\left(\frac{1}{n}\right), \qfor n \in \Z \text{ as } |n| \rightarrow \infty.
\end{equation}
Hence, when $|n|$ is large, the preimages of the positive real axis are close to circles $A_n$, where, for $n \in \Z$, $A_n$ is given by
\[
A_n \defeq \left\{ z = x + iy \in \C : x^2 + (y-p)^2 = p^2 \text{ where } p = \frac{1}{2n\pi} \right\}.
\]

Suppose that $z_0, z_1$ and $\gamma$ are as in the statement of the lemma. We can assume that these all lie in the upper half-plane. Recall that $\gamma$ cannot meet a preimage of the real line. Hence, if $R$ is sufficiently large, then $z_0$ and $z_1$ must lie in a domain which is very close to being the crescent between $A_n$ and $A_{n+1}$, for some large value of $n$. 

It follows from \eqref{eq:est2} that all points in this crescent have modulus that is less than $\frac{1}{n\pi} + O\left(\frac{1}{n^2}\right)$. By a geometric calculation, we find that all points in this crescent of modulus less than $\frac{1}{n\pi\sqrt{17}} + O\left(\frac{1}{n^2}\right)$ lie in $C(R)$. (The constant $17$ here follows from the fact that $\epsilon_0 = 1/4$.) The result then follows with $L_0 = 2\sqrt{17}$. 
\end{proof}
The second lemma is similar in many respects to the first, though the proof is more complicated.
\begin{lemma}
\label{lemm:longcurvesnearinfinity}
There exists $r_\infty > 0$ and $L_\infty > 1$ with the following property. Suppose that $R \geq r_\infty$, that $L \geq L_\infty$, that $K>1$, and that $z_0, z_1 \in \Cstar \setminus f^{-2}(\domain)$ are two points of modulus greater than $R$, such that $|z_1/z_0| \geq LK$. Suppose also that $\gamma \subset \Cstar \setminus f^{-2}(\domain)$ is a curve from $z_1$ to $z_0$. Then there is a curve $\gamma' \subset \gamma \cap C(R)$ joining endpoints $z_0'$ and $z_1'$, with $|z_1'/z_0'| \geq K$. More precisely, $\gamma'\subset C_n(R)$ for some $n\in \Z$.
\end{lemma}
\begin{proof}
We begin by estimating solutions to \eqref{eq:inV} for points $z = x + iy$ where $|y/x|$ is small; see Figure~\ref{f1}. Observe that when this is the case, solutions to \eqref{eq:inV} can be written in the form
\[
\frac{e^{-x}}{x^2 + y^2}(x \sin y  + y \cos y) = n \pi + O(|y/x|), \qfor n \in \Z.
\]
It follows by the compound angle formula that
\begin{equation}
\label{eq:inV3}
\sin(y + \arctan(y/x)) = e^x \cdot \sqrt{x^2 + y^2} \cdot \left(n\pi + O(|y/x|)\right), \qfor n \in \Z.
\end{equation}

%For $n \in \Z$, let $A_n'$ denote the set of points $z = x + iy$ that are solutions to \eqref{eq:inV3}. 
%
Hence when $n$ is large compared to $|y/x|$, the preimages of the positive real axis are close to the sets $A_n'$ where, for $n \in \Z$,
\[
A_n' \defeq \left\{ z = x + iy \in \C\, :\, \sin(y + \arctan(y/x)) = e^x \cdot \sqrt{x^2 + y^2} \cdot n\pi \right\}.
\]

When $n \ne 0$, no point on $A_n'$ (and indeed no preimage of the real line) can meet the set of points that are solutions to \eqref{eq:cantcross}, which we know are approximately the horizontal lines $\{ z=x + iy \in \C\, :\, y = m\pi\}$, for some $m \in \Z$. 

For all sufficiently large (and negative) values of $x$, the solution to \eqref{eq:inV3} is close to $y = m\pi$ for some $m \in \Z$. Note that points of large negative real part, and imaginary part close to an even multiple of $\pi$ map near the origin under $f$, and points of large negative real part, and imaginary part close to an odd multiple of $\pi$ map near infinity under $f$. We can deduce that $A_n'$ is a curve that  starts in the far left asymptotic to $y = m\pi$, travels some distance to the right, and then returns to the left asymptotic to $y = (m+1)\pi$. In particular, since $|\sin t| \leq 1$, for $t \in \R$, the rightmost point of $A_n'$ is where
\[
e^x \cdot \sqrt{x^2 + y^2} \cdot \left(n\pi + O(|y/x|)\right) = \pm 1.
\]
For large values of $|n|$, this gives $x = x_n = -\log |n| + O(\log \log |n|)$. Observe that there exists $\ell > 0$, independent of $n$, such that for all sufficiently large values of $|n|$, if $x + iy \in A_n'$ and $x < x_n - \ell$, then $x + iy \in C(R)$.
 
Suppose that $K, z_0, z_1$ and $\gamma$ are as in the statement of the lemma, and let $L>0$ be large. Let $\delta > 0$ be small. It follows from Lemma~\ref{lem:imag parts} that if $L$ is sufficiently large, then there is a subcurve $\tilde{\gamma}$ of $\gamma$, which joins $z_1$ to a point $\tilde{z_0}$ with $|z_1/\tilde{z_0}| \geq LK/2$, and such that if $z = x + iy \in \tilde{\gamma}$, then $|z| > R$ and $|y/x| < \delta$. We can assume that $\tilde{\gamma}$ lies in the upper half-plane.

Recall that $\tilde{\gamma}$ cannot meet the preimages of the positive real line. In the part of the plane where $\tilde{\gamma}$ lies we know that $|y/x|$ is small. So these points are close to the solutions to \eqref{eq:inV3}. In other words $\tilde{\gamma}$ can be assumed to lie in the region very close to that between (in the obvious sense) $A_n'$ and $A_{n+1}'$, for some large value of $|n|$.

As noted earlier, all points in this region have real part less than $$x_n = -\log |n| + O(\log \log |n|).$$ All points in this region of real part less that $x_n - \ell$ lie in $C(R)$. The result of the lemma holds, for sufficiently large $r_\infty > 0$, with (for example) $L_\infty = 4(\ell + 2\pi)$.
\end{proof}

We are now able to complete the proof, by showing that any component of $\Cstar \setminus I$ is bounded in $\Cstar$. To prove this, suppose by way of contradiction that $X$ is such a component that is not bounded in $\Cstar$. Let $K$ be the constant from Lemma~\ref{lemm:bigplaces}. First we choose
\[
L > \max \{ 2, L_0, L_\infty \},
\]
where $L_0$ is the constant in Lemma~\ref{lemm:longcurvesnearzero} and $L_\infty$ is the constant in Lemma~\ref{lemm:longcurvesnearinfinity}. For this value of $L$, we let $R_0$ be the constant from Lemma~\ref{lemm:bigplaces}. We then choose
\[
R > \max \{ 2, R_0, r_0, r_\infty \},
\] 
where $r_0$ is the constant in Lemma~\ref{lemm:longcurvesnearzero} and $r_\infty$ is the constant in Lemma~\ref{lemm:longcurvesnearinfinity}.

For $\rho > 1$, we use the notation
\[
D_\rho \defeq \{ z \in \Cstar : |z| \geq \rho \ \text{ or } |z| \leq 1/\rho \}.
\]
When $\rho$ is large, points of $D_\rho$ have either large modulus or small modulus.

We construct sequences of points $(z^0_k)_{k\geq 0}$ and $(z^1_k)_{k\geq 0}$, and curves $(\gamma_k)_{k\geq 0}$, with $\gamma_0 \subset X$, and with the following properties, for each $k \geq 0$:
\begin{enumerate}[(I)]
\item $\gamma_k \subset C(R) \cap D_{2^kR}$ is a curve from $z^0_k$ to $z^1_k$. \label{en:gamma}
\item $|z^1_k/z^0_k| \geq K$. \label{en:size}
\item $\gamma_{k+1} \subset f(\gamma_k)$. \label{en:image}
\end{enumerate}

Note that \eqref{en:gamma} says that $\gamma_k$ is in one of the channels that make up $C(R)$ and, in addition, is either near the origin, or far from the origin (in which case it is also far to the left). % Note also that \eqref{en:size} implies that, for large values of $k$, $z^1_k$ and $z^0_k$ have very different moduli, and so $\gamma_k$ is ``long'' in the sense mentioned earlier. % (Near infinity ``long'' means long, whereas near the origin ``long'' means that the image of the curve under the map $z \mapsto 1/z$ is long).

Note also that this construction completes the proof. To see this, observe that \eqref{en:image}, together with Lemma~\ref{lem:slow}, imply the existence of a point in $z \in \gamma_0 \subset X$ which satisfies $f^k(z) \in \gamma_k$. This in turn implies, by \eqref{en:gamma}, that $z \in I$, which is a contradiction.

The construction is by induction. The start of the induction is straightforward; it follows from Lemma~\ref{lemm:longcurvesnearzero} and Lemma~\ref{lemm:longcurvesnearinfinity}, together with our assumption that $X$ is unbounded in $\Cstar$, that we can choose $z^0_0$, $z^1_0$ and $\gamma_0$ with properties \eqref{en:gamma} and \eqref{en:size}.

Now suppose that $k \geq 0$, and we have chosen $z^0_k$, $z^1_k$ and $\gamma_k$ with all the required properties. Set $\Gamma_k = f(\gamma_k)$. It follows from Lemma~\ref{lemm:bigplaces}, and the fact that $L>2$, that $\Gamma_k \subset D_{2^k R}$. Moreover, it is a consequence of Observation~\ref{obs:eventually in I} that $\Gamma_k$ cannot meet $f^{-2}(\domain)$.

Since $|z^1_k/z^0_k| \geq K$, we have by Lemma~\ref{lemm:bigplaces} that $$\max\{|f(z^1_k)/f(z^0_k)|, |f(z^0_k)/f(z^1_k)|\} \geq LK.$$ 

We then complete the construction with an application of either Lemma~\ref{lemm:longcurvesnearzero} if $\Gamma_k \subset \{ z \in \Cstar : |z| \leq 1/(2^kR) \}$, or Lemma~\ref{lemm:longcurvesnearinfinity} if $\Gamma_k \subset \{ z \in \Cstar : |z| \geq 2^k R \}$. This completes the construction, and so completes the proof.

%\end{proof}

%
%%%%

%
%%%%
%
%\bibliographystyle{acm}
\bibliography{Spiderswebs}
\end{document}